\documentclass[reqno]{amsart}
\usepackage{graphicx}
\usepackage{amssymb,amscd,amsthm,amsxtra}
\usepackage{latexsym}
\usepackage{epsfig}
\usepackage{mathtools}
\usepackage{color}
\usepackage{amsaddr}
\usepackage{enumerate}

\newtheorem{thm}{Theorem}[section]
\newtheorem{cor}[thm]{Corollary}
\newtheorem{lem}[thm]{Lemma}
\newtheorem{prop}[thm]{Proposition}
\theoremstyle{definition}

\theoremstyle{remark}
\newtheorem{rem}[thm]{Remark}
\numberwithin{equation}{section}

\newcommand{\be}{\begin{equation}}
\newcommand{\ee}{\end{equation}}

\newcommand{\comment}[1]{}

\begin{document}

\title[]{
Regularity properties for a class of \\  non-uniformly elliptic Isaacs operators
}

\author{Fausto Ferrari}
\address{Dipartimento di Matematica dell'Universit\`a di Bologna, \\
Piazza di Porta S. Donato, 5, 40126 Bologna, Italy.}
\email{\tt fausto.ferrari@unibo.it}
\author{Antonio Vitolo}
\address{\hskip-0.5cm Dipartimento di Ingegneria Civile, Universit\`a di Salerno, \\
Via Giovanni Paolo II, 132 
- 84084 Fisciano (SA), Italy,
 \\and\\
 Istituto Nazionale di Alta Matematica, INdAM - GNAMPA, Italy.}
\email{\tt vitolo@unisa.it} 


\thanks{F.F. is partially funded by INDAM-GNAMPA project 2018: {\it Costanti critiche e problemi asintotici per equazioni completamente non lineari} and INDAM-GNAMPA project 2019:  {\it Propriet\`a di regolarit\`a delle soluzioni viscose con applicazioni a problemi di frontiera libera.}}
{\begin{abstract} We consider the elliptic differential operator defined as the sum of the minimum and the maximum eigenvalue of the Hessian matrix, which can be viewed as a degenerate elliptic Isaacs operator, in dimension larger than two. Despite of  nonlinearity, degeneracy, non-concavity and non-convexity, such operator generally enjoys the qualitative properties of the Laplace operator, as for instance maximum and comparison principles, ABP and Harnack inequalities, Liouville theorems for subsolutions or supersolutions. Existence and uniqueness for the Dirichlet problem are also proved as well as local and global H\"older estimates for viscosity solutions. All results are discussed for a more general class of weighted partial trace operators.



\end{abstract}}

\maketitle

\section{Introduction and main results}\label{intro}

In this paper we investigate the properties of weighted partial trace operators 
\begin{equation}\label{Ma}
\mathcal M_{\bf a}:=\sum_{i=1}^{n}a_i\lambda_i(X)
\end{equation}
where $\lambda_i(X)$ are the eigenvalues  of $X\in \mathcal S^n$, the set  of $n \times n$ real symmetric matrices,  in increasing order, that is
\begin{equation}\label{eigenv}
\lambda_1(X) \le \dots \le \lambda_n(X),
\end{equation}
and ${\bf a}=(a_1,\dots,a_n)\in \mathbb R^n$ is a n-ple of non-negative coefficients $a_i$ such that $a_j>0$ for at least one $j \in\{1,\dots,n\}$.

The class $\overline{\mathcal A}$ of such operators includes the partial trace operators
\begin{equation}\label{ptrace}
\begin{split}
 \mathcal P_k^-(X)=\sum_{i=1}^k\lambda_i(X), \quad \mathcal P_k^+(X)= \sum_{i=n-k+1}^n\lambda_i(X),
\end{split}
\end{equation}
considered by Harvey-Lawson \cite{HL}, \cite{HL2} and Caffarelli-Li-Nirenberg \cite{CLN}, \cite{CLN2}.

Here we introduce the subclass $\mathcal A$, characterized by non-negative coefficients $a_i$ such that $a_1>0$ and $a_n>0$, which in some sense complements the set of operators $\mathcal P^\pm_k(X)$ with $k<n$. In fact, the  prototype of $\mathcal A$ is the min-max operator
$$
\mathcal M(X):=\lambda_1(X)+\lambda_n(X).
$$
As we will see later, $\mathcal M$ can be in fact  viewed as a degenerate elliptic Isaacs operator (for $n \ge 2$) whereas $\mathcal P^\pm_k(X)$ results in a degenerate elliptic Bellman operator (for $k<n$).

Of course, the case $n=2$ is by far well known, because $\mathcal M$ reduces to the classical Laplace operator. However, in higher dimension, namely for $n>2,$ the operator ceases to be uniformly elliptic, it becomes a fully nonlinear non-convex degenerate elliptic operator. Nonetheless, we will see,  rather surprisingly, that it retains many properties of the Laplace operator.

It also worth noticing that the operators $\mathcal M_{\bf a}$ of the smaller subclass  $\underline{\mathcal A}$,  characterized by weights $a_i>0$ for all $i=1,\dots,n$, are uniformly elliptic, as we will see in Section \ref{prelim}.

\vskip0.2cm

After introducing our main results we shall further come back to the original motivation for studying the operators of $\overline{\mathcal A}$, and in particular the subclass $\mathcal A$.

A good number of results will depend on the dimension $n$ and on the following two quantities, namely the minimum between the coefficients of the smallest and the greatest eigenvalue, and the arithmetic mean of the coefficients, namely
\begin{equation}\label{anumbers}
a^*:=\min(a_1,a_n), \quad \tilde a:= \frac{a_1+\dots+a_n}{n},
\end{equation}
in the sense that the involved constants are uniformly bounded when a positive upper bound of the first one and a finite upper bound of the latter one are avalaible. 

The constants which depend only on $n$, $a_1$, $a_n$ and $ \tilde a$ will be also called universal constants.

\vskip0.2cm

The following result is a revisitation of the bilateral Alexandroff-Bakelman-Pucci estimate (ABP) for the class $\mathcal{A}$, only depending on $n$ and $a^*$. 

\begin{thm}\label{ABP-two}
Let $\Omega\subset \mathbb{R}^n$ be a bounded open domain of diameter $d$. Let  $f\in C(\Omega)$ be bounded in $\Omega$.  If $u\in C(\bar{\Omega})$  is a viscosity solution to the equation $\mathcal M_{\bf a}(D^2u)=f$  in $\Omega$, with $\mathcal M_{\bf a} \in \mathcal A,$  then:
\begin{equation}\label{ABP-est-two}
\sup_{\Omega}|u|\leq \sup_{\partial\Omega}u^+ + \frac{C_n}{a^*}\,d\,\|f\|_{L^n(\Omega)},
\end{equation}
where $C_n>0$ is a positive constant depending only on $n$.
\end{thm}

We emphasize the following difference between (\ref{ABP-est-two}) and the standard ABP estimates, see for instance \cite[Theorem 9.1]{GT}:  the denominator of the right-hand side is $a^*=\min(a_1,a_n)$  instead of the geometric mean $\mathcal D^*=(a_1\cdots a_n)^{1/n}$, the geometric mean of the coefficients, which would be useless  in the non-uniformly elliptic case, as soon as one of the $a_j$'s is zero,  while $a^*$ is positive for the class $\mathcal A$. 

The above result is obtained as a consequence of two unilateral ABP estimates for subsolutions (\ref{ABP-est-sub}) and supersolutions (\ref{ABP-est-super}).  

\vskip0.2cm

The ABP estimate stated before also underlies a corresponding Harnack inequality for the equation  $\mathcal M_{\bf a}[u]=f$, depending on $n$,  $a^*$ and $ \tilde a$, instead of the elliptic constants $\lambda$ and $\Lambda\ge \lambda$, which would be ineffective in the degenerate elliptic case in which $\lambda =0$. This Harnack inequality cannot be extended to arbitrary degenerate elliptic operators of the class $\overline{\mathcal A}$, and in particular it fails to hold for partial trace  equations $\mathcal P^\pm_k[u]=f$ when $k<n$. 

\begin{thm}\label{Harnack} {\rm(Harnack inequality)}  Let $\mathcal M_{\bf a}\in \mathcal A$. 
Let $u$ be a viscosity solution of the equation $\mathcal M_{\bf a}(D^2u)=f$ in the unit cube $Q_1$ such that $u \ge 0$ in $Q_1$, where $f$ is continuous and bounded. Then
\begin{equation}\label{harnack-ineq}
\sup_{Q_{1/2}} u \le C\left( \inf_{Q_{3/4}}u+ \|f\|_{L^n(Q_1)}\right),
\end{equation}
where $C$ is a positive constant depending only on $n$, $a^*$ and  $\tilde a$.
\end{thm}

We prove Theorem \ref{Harnack}
via two inequalities for subsolutions and non-negative supersolutions, known in literature respectively as the local maximum principle (Theorem \ref{lmp}) and the weak Harnack inequality (Theorem \ref{w-H}), suitably adapted to this framework, by comparison with Pucci extremal operators.

From the Harnack inequality, the interior $C^\alpha$ estimates of Theorem \ref{Holder} in Section \ref{harnack} follow with a universal exponent $\alpha\in(0,1)$, in the same way of the uniformly elliptic case \cite{CC}.

 Here, in Lemma \ref{Holder-bdary-lem},  we get  boundary H\"older   estimates assuming for $\Omega$ a uniform exterior sphere property, with radius $R>0$: 

(S) {\it for all $y\in \partial\Omega$ there is a ball $B_R$ of radius $R$ such that $y\in \partial B_R$ and $\overline\Omega\subset \overline B_R$}.

\vskip0.2cm

We obtain the following  estimates  for the H\"older  seminorm $[u]_{\gamma,\Omega}$. See the notation (\ref{betanorm}) in Section \ref{harnack}.

\begin{thm} \label{Holder-global} {\rm (global H\"older estimates)} Let $u\in C(\overline \Omega)$ be a viscosity solution of the equation $\mathcal M_{\bf a}(D^2u) =f$  in a bounded domain $\Omega$. 

We assume that with $\mathcal M_{\bf a}\in \mathcal A$ and  $f$ is continuous and bounded in $\Omega$. Let also $\alpha\in (0,1)$ be the exponent of the interior $C^\alpha$ estimates.

(i) Suppose that $\Omega$ satisfies a uniform exterior sphere condition (S) with radius $R>0$. If $u=g$ on $\partial \Omega$ with $g \in C^{\beta}(\partial \Omega)$ and $\beta \in (0,1]$, then $u\in C^{\gamma}(\overline \Omega)$ with $\gamma=\min(\alpha, \beta/2)$, and 
\begin{equation}\label{global-Holder1-intro}
[u]_{\gamma,\Omega} \le C\left(\|g\|_{C^{\beta}(\partial\Omega)}+\|f\|_{L^\infty(\Omega)}\right),
\end{equation}
where $C$ is a positive constants depending only on $n$, $a^*$,  $\tilde a_n$, $R$, $L$ and $\beta$.

(ii) Suppose in addition that $\Omega$ has a uniform Lipschitz boundary with Lipschitz constant $L$. If $g\in  C^{1,\beta}(\partial \Omega)$ with $\beta \in [0,1)$, then $u\in C^{(1+\beta)/2}(\overline \Omega)$, where $\gamma=\min(\alpha,(\beta+1)/2)$, and
\begin{equation}\label{global-Holder2-intro}
[u]_{(1+\beta)/2,\Omega} \le C\left(\|g\|_{C^{1,\beta}(\partial\Omega)}+\|f\|_{L^\infty(\Omega)}\right) .
\end{equation}
\end{thm}

A global estimate for the H\"older norm $\|u\|_{C^{0,\gamma}(\Omega)}= \|u\|_{L^{\infty}(\Omega)}+[u]_{\gamma,\Omega}$ can be obtained combining the above estimates with the uniform estimate of Corollary \ref{unif-est-cor}. 

\vskip0.2cm
 In some cases, we can obtain an explicit interior H\"older exponent. For instance, in the case of  asymmetric distributions of weights, concentrated on the smallest or the largest eigenvalue, as for the upper and lower partial trace operators $\mathcal P^\pm_k$, $k < n$,  see Lemma \ref{Holder-asy-int}.  The result depend in fact on the smallness of the quotients $\hat a_1/a_1$ and $\hat a_n/a_n$ (see Subsection \ref{Radial_solutions}), as it can be seen in the statement below.

\begin{thm}\label{Holder-asy-global}  Let $u\in C(\overline \Omega)$ be a viscosity solution of the equation $\mathcal M_{\bf a}(D^2u) =f$  in a bounded domain $\Omega$, where $f$ is continuous and bounded. Suppose $\mathcal M_{\bf a}\in \overline{\mathcal A}$ with $a_1 \ge \hat a _1$, resp. with $a_n\ge \hat a _n$.  Then the global H\"older estimates of Theorem \ref{Holder-global} hold, namely (\ref{global-Holder1-intro}) in the case (i) and (\ref{global-Holder2-intro}) in the case (ii), with
$$\alpha =\max(1-\hat a_1/a_1, 1-\hat a_n/a_n).$$

In particular, we deduce the following $C^\alpha$ estimates.
\vskip 0.2cm
(i) Suppose that $\Omega$ satisfies a uniform exterior sphere condition (S) with radius $R>0$. If $u=g$ on $\partial \Omega$ with $g \in C^{2\alpha}(\partial \Omega)$ and $\alpha \in (0,1/2]$, then $u\in C^{\alpha}(\overline \Omega)$, and 
\begin{equation}\label{global-Holder1-asy}
[u]_{\alpha,\Omega} \le C\left(\|g\|_{C^{2\alpha}(\partial\Omega)}+\|f\|_{L^\infty(\Omega)}\right),
\end{equation}
where $C$ is a positive constant depending only on $n$, $a_1$, $\hat a_1$, $a_n$, $\hat a_n$,  $\tilde a$, $R$ and $\alpha$. 
\vskip 0.2cm

(ii) Suppose in addition that $\Omega$ has a uniform Lipschitz boundary with Lipschitz constant $L$. If $g\in  C^{1,2\alpha-1}(\partial \Omega)$ with $\alpha \in [1/2,1]$, then 
\begin{equation}\label{global-Holder2-asy}
[u]_{\alpha,\Omega} \le C\left(\|g\|_{C^{1,2\alpha-1}(\partial\Omega)}+\|f\|_{L^\infty(\Omega)}\right),
\end{equation}
where $C$ is a positive constant also depending on $L$. 
\end{thm}

The regularity issue is far from being completely explored in the case of degenerate, non-uniform ellipticity. See for instance \cite{IMB,BCDV} for other kind of singular or degenerate elliptic operators and \cite{F,FV} for non-commutative structures. 

Concerning higher regularity, one could borrow the techniques of \cite{TRU}, \cite{CAF}, \cite{CC2}, \cite{SWI}, \cite{CDV0}, \cite{CDV}, \cite{KOV1}, \cite{IS}, \cite{KOV2}, which however do not seem at the moment directly applicable  in the more general non-uniformly elliptic setting.

It is remarkable the particular case of the interior $C^{1,\alpha}$ regularity proved in \cite{OS} for the equation $\lambda_1[u]= f(x)$ with $C^{1,\alpha}$ boundary data. 

\vskip0.2cm

Further aspects of the qualitative theory, like the strong maximum principle and Liouville theorems, will be discussed in the last sections of the paper. New results for operators $\mathcal M_{\bf a}\in \overline{\mathcal A}$ will be shown there, depending on the relative magnitude of $ a_1$ and $a_n$, and their complements $\hat a_1$ and $\hat a_n$ with respect to $|{\bf  a}|=a_1+\dots+a_n$, see Subsection \ref{Radial_solutions}. 

\vskip0.2cm

Turning to the motivations about the importance of this research, we recall that the partial trace operators $\mathcal{P}^+_k(X)$ are degenerate elliptic operators, which can be represented as Bellman operators:
\begin{equation}\label{ptraceb}
\begin{split}
\mathcal{P}^+_k(X)= \sup_{W\in\mathcal G_k}\hbox{\rm Tr}(X_W), \quad \mathcal{P}^-_k(X)= \inf_{W\in\mathcal G_k}\hbox{\rm Tr}(X_W),
\end{split}
\end{equation}
where $\mathcal G_k$ is the Grassmanian of the $k$-dimensional subspace $W$ of $\mathbb R^n$ and $X_W$ is tha matrix of the quadratic form associated to $X$ restricted to $W$, see \cite{HL}. 

Upper and lower partial trace operators arise in geometric problems of mean partial curvature considered by Wu \cite{WU} and Sha \cite{SHA,SHA2}. Following the interest generated by the previous works, a number of papers has been devoted to the properties of these operators, we recall for instance \cite{AGV,CDLV,GV,VIT2}.   

On the other hand, it is also worth noticing that Bellman equations arise in stochastic control problem, see Krylov \cite{KRY}, Fleming-Rishel \cite{FR}, Fleming-Soner \cite{FS} and the references therein.

\vskip0.2cm

As well as the partial trace operators $\mathcal P^\pm_k$ with $k<n$ constitute a model for degenerate elliptic Bellman operators, the min-max operator $\mathcal M$ provides for $n \ge 3$ a prototype of degenerate elliptic Isaacs operators by the representation:
\begin{equation}\label{infsup}
\mathcal M(X) = \sup_{|\xi| =1}\inf_{|\eta| =1} \hbox{\text Tr}(X_{\xi,\eta}),
\end{equation}
where Tr$(X_{\xi,\eta})$ is the trace the matrix $X_{\xi,\eta}$ of the quadratic form associated to $X$ restricted to $L(\xi,\eta)$, the subspace of $\mathbb R^n$ spanned by $\xi$ and $\eta$.
  
The alternative representation
\begin{equation}\label{M-quadratic}
\mathcal M(X) = \max_{|\xi|=1} {\langle X\xi,\xi\rangle} + \min_{|\xi|=1} {\langle X\xi,\xi\rangle},
\end{equation}
suggests the relationship between $\mathcal M(X)$ and stochastic zero-sum, two-players differential games and  Isaacs equations,
for which  we we refer for instance to \cite{NIS, FN, FS} and the references therein and to \cite{BL, BCQ} for more recent contributions. 

\vskip0.2cm

Following the main stream of the mean value properties of solution to linear equations, as well as in the case of the $\infty$-Laplacian, it is also worth to be remarked that  whenever $u$ is $C^2,$ the following expansion yields 
\begin{equation}\label{Taylor}
\begin{split}
&u^\varepsilon_1(x)\equiv \min_{|\xi|=1}\frac{u(x+\epsilon \xi)+u(x-\epsilon \xi)}{2} = u(x)+ \frac{\varepsilon^2}2\,\lambda_1(x)+o(\varepsilon^2), \\ 
&u^\varepsilon_2(x)\equiv \max_{|\eta|=1}\frac{u(x+\epsilon \eta)+u(x-\epsilon \eta)}{2} = u(x)+ \frac{\varepsilon^2}2\,\lambda_n(x) +o(\varepsilon^2).
\end{split}
\end{equation}
As a consequence, if we consider a continuous function $u,$ the operator given by the following limit, 
\begin{equation}\label{mean}
\begin{split}
\lim_{\epsilon\to 0}\tfrac{2}{\epsilon^2}\left( u_{\varepsilon,1}(x)+ u_{\varepsilon,2}(x)-2u(x)\right),
\end{split}
\end{equation}
whenever it exists, may be considered as the weak version of our operator $\mathcal{M}$. For an almost compete list of references from this point of view see: \cite{NV,MPR} for the p-Laplace equation, as well as \cite{FLM,FP}  for further applications to non-commutative fields where a lack of ellipticity occurs. 

\vskip0.2cm

The paper is organized as follows. In Section \ref{prelim} we introduce the main definitions about elliptic operators and viscosity solutions. In Section \ref{aux} we discuss in detail the properties of the weighted partial trace operators, in particular $\mathcal M$. We show a comparison principle, an existence and uniqueness theore, and compute the radial solutions. In Section \ref{abp} we prove Theorem \ref{ABP}. In Section \ref{harnack} we show the Harnack inequality, interior and boundary H\"older estimates. We also discuss, in Section \ref{strong}, the strong maximum principle via  both the Hopf boundary point lemma and the Harnack inequality, showing suitable counterexamples. Finally, in Section \ref{liouville}, we also prove  Liouville theorems and an unilateral Liouville property with the Hadamard's three circles theorem.

\section{General preliminaries}\label{prelim}

This section is organized in some subsections, mainly for introducing common notation about viscosity theory of elliptic nonlinear PDEs, see Subsection \ref{visco1} and \ref{dua5}.  In Subsections \ref{operator2} and \ref{minmax3}  we introduce our class of operators $\mathcal{A}$ and in particular discuss  the min-max operator $\mathcal{M}$ showing by counterexamples that it is nonlinear, non-convex and non-uniformly elliptic. In Subsection \ref{compar4} we discuss a comparison result with the partial trace operators operator $\mathcal P^{\pm}_k.$ 

\subsection{Ellipticity and viscosity solutions}\label{visco1}

We start recalling some ellipticity notions. Let $\mathcal{S}^{n}$ be the set of $n \times n$ symmetric matrices with real entries, partially ordered with the relationship $X \le Y$ if and only if $Y-X$ is semidefinite positive. 

\vskip0.2cm
A fully nonlinear operator, that is a mapping  $\mathcal F: \mathcal{S}^{n} \to \mathbb R$, is said degenerate elliptic if 
\begin{equation}\label{de} 
X \le Y \ \ \Rightarrow \ \ \mathcal F(X) \le \mathcal F(Y),
\end{equation}
and uniformly elliptic if
\begin{equation}\label{ue}
X \le Y \ \ \Rightarrow \ \  \lambda {\text Tr}(Y-X) \le \mathcal F(Y)-\mathcal F(X) \le  \Lambda {\text Tr}(Y-X),
\end{equation}
for positive constants $\lambda$ and $\Lambda$, called ellipticity constants.
Note indeed that, by the left-hand side inequality in (\ref{ue}), a uniformly elliptic operator $\mathcal F$ satisfies (\ref{de}), and so it is degenerate elliptic. 

The uniform ellipticity also implies the continuity of the mapping $\mathcal F:{\mathcal S}^n \to \mathbb R$. In what follows we also assume that $\mathcal F$ is a continuous mapping even in the degenerate elliptic case.

\vskip0.2cm

Suppose now $X \le Y$. It is plain that ${\text Tr}(Y-X)\ge0$. Suppose in addition $\mathcal F(Y)=\mathcal F(X)$.  If $\mathcal F$ is uniformly elliptic, in view of the left-hand side of (\ref{ue}), we also have ${\text Tr}(Y-X)\le0$, so that ${\text Tr}(Y-X)=0$. Then $Y=X$. In other words, $\mathcal F$ is strictly increasing on ordered chains of $\mathcal S^n$.

\vskip0.2cm

The class of uniformly elliptic operators with given ellipticity constants $\lambda$ and $\Lambda$ is bounded by two estremal operators, the maximal and minimal Pucci operator, which are in turn uniformly elliptic with the same ellipticity constants, respectively:
\begin{equation}\label{Pucci}
\begin{split}
&\mathcal M_{\lambda,\Lambda}^+(X)=\Lambda \hbox{\text Tr}(X^+)-\lambda \hbox{\text Tr}(X^-),\\
&\mathcal M_{\lambda,\Lambda}^-(X)=\lambda \hbox{\text Tr}(X^+)-\Lambda \hbox{\text Tr}(X^-),
\end{split}
\end{equation}
where $X=X^+-X^-$ is the unique decomposition of $X \in \mathcal S^n$ as difference of semidefinite positive matrices $X^+$ and $X^-$ such that $X^+X^-=0$.

In view of this definition, the uniformly ellipticity (\ref{ue}) of $\mathcal F$ can be equivalently stated as
\begin{equation}\label{uee}
\forall\,X,Y\in\mathcal S^n \quad \quad   \mathcal M_{\lambda,\Lambda}^-(Y-X) \le \mathcal F(Y)-\mathcal F(X) \le   \mathcal M_{\lambda,\Lambda}^+(Y-X).
\end{equation}

From this it also follows that, if $\mathcal F$ is uniformly elliptic and $\mathcal F(0)=0$, then
\begin{equation}\label{ue0}
\forall\,X\in\mathcal S^n \quad \quad   \mathcal M_{\lambda,\Lambda}^-(X) \le \mathcal F(X) \le   \mathcal M_{\lambda,\Lambda}^+(X),
\end{equation}
which shows the extremality of Pucci operators.

\vskip0.2cm 

Throughout this paper we will assume in fact
\begin{equation}\label{F(0)=0}
\mathcal F(0)=0.
\end{equation} 

Of course, the results can be applied, in the case $\mathcal F(0)\neq 0$, to the operator $\mathcal G(X)=\mathcal F(X)- \mathcal F(0)$.

\vskip0.2cm

Let $\Omega$ be an open set of $\mathbb R^n$. A fully nonlinear operator $\mathcal F$ acts on $u \in C^2(\Omega)$ through the Hessian matrix $D^2u$ setting 
$$\mathcal F[u](x)=\mathcal F(D^2u(x)).$$

Let $f$ be a function defined in $\Omega$. A solution $u \in C^2(\Omega)$ of the equation $\mathcal F[u]=f$ is called a classical solution, as well as classical subsolution or supersolution of $F[u]=f$ if $\mathcal F(D^2u(x))\ge f(x)$ or $\mathcal F(D^2u(x))\le f(x)$ for every $x \in\Omega$, respectively.

 For instance, if $\mathcal F(X)=$Tr$(X)$ and $f(x)$ is a continuous function, then $\mathcal F[u]=\Delta u$ is the Laplacian and the equation  $\mathcal F[u]=f$ is the Poisson equation $\Delta u=f$. 

\vskip0.2cm

Let $\mathcal F$ be a degenerate elliptic operator. We can solve the equation $\mathcal F[u]=f(x)$ in a weaker sense, namely in the viscosity sense. We are essentially concerned in this paper with pure second order operators $\mathcal F[u]=\mathcal F(D^2u)$. We refer to \cite{CC} and \cite{USE} for general operators, also depending on $x \in \Omega$, $u$ and the gradient $Du$, and to \cite{HL} for a geometric interpretation of viscosity solutions.

We briefly recall what means to solve the equation $\mathcal F[u]=f$ introducing sub/super jets basic notions.

\vskip0.2cm

Let $\mathcal O$ be a locally compact subset of $\mathbb R^n$, and $u:\mathcal O \to \mathbb R$.  The second order superjet $J_{\mathcal O}^{2,+}u(x_0)$ and subjet $J_{\mathcal O}^{2,-}u(x_0)$ of $u$ at $x_0\in \mathcal O$ are respectively the sets 
\begin{align*}
J_{\mathcal O}^{2,+}u(x_0) &= \left\{(\xi,X) \in   \mathbb R^n \times \mathcal S^n  : u(x) \le u(x_0)+\langle\xi,x-x_0\rangle\right. \\
& \left.+\tfrac12\langle X(x-x_0),(x-x_0)\rangle+ o(|x-x_0|^2) \ \ \hbox{\text as} \ x \to x_0\right\}
\end{align*}  
and 
\begin{align*}
J_{\mathcal O}^{2,-}u(x_0) &= \left\{(\xi,X) \in \mathbb R^n \times \mathcal S^n   : u(x) \ge u(x_0)+\langle\xi,x-x_0\rangle\right. \\
& \left.+\tfrac12\langle X(x-x_0),(x-x_0)\rangle+ o(|x-x_0|^2) \ \ \hbox{\text as} \ x \to x_0\right\}.
\end{align*} 

 We denote by usc$(\mathcal O)$  and lsc$(\mathcal O)$ the set of upper and lower semicontinuous functions in $\mathcal O$, respectively. 

If $u$ is  usc$(\mathcal O)$, then $u$ is a viscosity subsolution of a fully nonlinear elliptic equation $\mathcal F[u]=f$ if 
$$\mathcal F(X)\ge f(x) \ \ \hbox{for all} \ x \in \mathcal O \ \ \hbox{and all}\:\: (\xi,X) \in J_{\mathcal O}^{2,+}u(x).$$

If $u\in$\,lsc$(\mathcal O)$, then $u$ is a viscosity supersolution  of the same equation if
$$\mathcal F(X)\le f(x) \ \ \hbox{for all} \ x \in \mathcal O \ \ \hbox{and all}\:\:(\xi,X) \in J_{\mathcal O}^{2,-}u(x).$$

A viscosity solution of the equation $\mathcal M[u]=f$ is both a subsolution and a supersolution $u \in C(\mathcal O)$.

It is worth noticing that classical solutions  are viscosity solutions. Viceversa, viscosity solutions  of class $C^2$ are in turn classical solutions. The same holds for subsolutions and supersolutions.

\subsection{The operator class $\mathcal{A}$}\label{operator2}

Let $\{{\bf e}_1, \dots,{\bf e}_n\}$ be the standard basis in $\mathbb R^n$, such that $({\bf e}_i)_j=\delta_{ij}$ for $i,j=1,\dots,n$, and $\lambda_i(X)$, $i=1,\dots,n$, be the eigenvalues of $X\in \mathcal S^n$ in non-decreasing order. 

Let ${\bf a}=(a_1,\dots,a_n)=a_1{\bf e}_1+\dots+a_n{\bf e}_n$. We consider the class of degenerate elliptic weighted trace operators 
\begin{equation}\label{basic-class}
\overline {\mathcal A}=\{\mathcal M_{\bf a} :  \underline a\equiv \min_i a_i \ge 0;  \ \overline a\equiv \max_i a_i>0\},
\end{equation} 
where 
\begin{equation*}
\mathcal M_{\bf a}(X) = a_1\lambda_1(X)+\dots+a_n\lambda_n(X), \ \ \hbox{\text see (\ref{Ma})}.
\end{equation*}

We observe that $\overline {\mathcal A}$ contains both uniformly and non-uniformly elliptic operators. In particular all previously considered operators belong to this class with a suitable representation:
\begin{equation}\label{examples}
\begin{split}
&{\text Tr}(X)=\mathcal M_{{\bf e}_1+\dots+{\bf e}_n}(X);\quad \quad  \mathcal M(X)=\mathcal M_{{\bf e}_1+{\bf e}_n}(X);\\
&\mathcal P^+(X)= \mathcal M_{{\bf e}_{n-k+1}+\dots+{\bf e}_n}(X); \quad \mathcal P^-(X)= \mathcal M_{{\bf e}_{1}+\dots+{\bf e}_k}(X).
\end{split}
\end{equation} 

Very recently, recalling the pioneeristic paper \cite{PSSW}, Blanc and Rossi \cite{BR} have shown  that it is possible to define a game satisfying a dynamic programming principle (DPP) which leads to the Dirichlet problem 
\begin{equation}\label{DPPlimit}\begin{split}
\left\{
\begin{array}{ll}
\mathcal{M}_{\bf a}[u]=0,&\hbox{\text in}\ \Omega\\
u=g(x),&\hbox{\text on}\ \partial\Omega\,.
\end{array}
\right.
\end{split}
\end{equation} 
Moreover, an associated evolution problem is considered in \cite{BER}.

\vskip 0.2cm

 We point out that $\mathcal M=\mathcal M_{{\bf e}_1+{\bf e}_n}$ is neither linear nor uniformly elliptic, neither concave nor convex, except when $n=2$, as it follows from the representation (\ref{infsup}) and it will be proved in the next section with suitable counterexamples.

\vskip0.2cm

Actually, $\mathcal M$ is a model of a larger class of degenerate, possibly non-uniformly elliptic operators:
\begin{equation}\label{A}
\mathcal A = \{\mathcal M_{\bf a} : \underline a \ge 0; \ a^*\equiv \min(a_1, a_n)>0\},
\end{equation}
which can be seen as $\mathcal A=\mathcal A_1\cap \mathcal A_n$, where 
\begin{equation}\label{Aj}
\mathcal A_j = \{\mathcal M_{\bf a} : \underline a \ge 0; \ a_j>0\}.
\end{equation}

Setting in addition
\begin{equation}\label{Aunif}
\underline{\mathcal A} = \{\mathcal M_{\bf a} :  \underline a  > 0\},
\end{equation}
we notice that   
$$ \underline{\mathcal A} \subset  \mathcal A = \mathcal A_1 \cap \mathcal A_n \subset \overline{\mathcal A}.$$ 

We remark for instance that, while the min-max operator $\mathcal M$ belongs to $\mathcal A$, the partial trace operators $\mathcal P^-_k \in \mathcal A_1$, $ \mathcal P^+_k \in \mathcal A_n$, do not belong to $\mathcal A$ for $k<n$.

\vskip0.2cm
On the other hand, every $\mathcal M_{\bf a} \in \underline{\mathcal A}$ is uniformly elliptic. 

In fact, if $X \le Y$, then 
\begin{equation}\label{Ma-elliptic}
\begin{split}
\mathcal M_{\bf a}(Y) - \mathcal M_{\bf a}(X) =&  \sum_{i=1}^na_i(\lambda_i(Y)-\lambda_i(X))\\
\ge&\, \underline a\, \hbox{\text Tr}(Y-X),
\end{split}
\end{equation}
so that every $\mathcal M_{\bf a} \in \overline A$ is degenerate elliptic. Since $X\le Y$ also implies
\begin{equation}\label{Ma-lip}
\begin{split}
\mathcal M_{\bf a}(Y) - \mathcal M_{\bf a}(X) =&  \sum_{i=1}^na_i(\lambda_i(Y)-\lambda_i(X))\\
\le&\, \overline a\,  \hbox{\text Tr}(Y-X),
\end{split}
\end{equation}
we conclude that  $\mathcal M_{\bf a} \in \underline A$ is uniformly elliptic with ellipticity constants $\lambda=\underline a \equiv \min_ia_i$ and $\Lambda=\overline a \equiv \max_ia_i$.

\vskip0.2cm

We also observe that the operators $\mathcal M_{\bf a} \in \overline {\mathcal A}$ are invariant by rotation, since $\mathcal M_{\bf a}(\mathcal R^TX \mathcal R)=\mathcal M(X)$ for all orthogonal matrices $\mathcal R$, and are positively homogeneous of degree one:
\begin{equation}\label{pos-hom}
\mathcal M_{\bf a}(\rho X)=\sum_{i=1}^na_i\lambda_i(\rho X)= \rho\sum_{i=1}^na_i\lambda_i(X)=\rho \mathcal M_{\bf a}(X), \ \ \rho\ge0.
\end{equation}
Next, we investigate more closely the peculiar properties of the min-max operator $\mathcal M(X)=\lambda_1(X)+\lambda_n(X)$.

\subsection{The min-max operator $\mathcal M$}\label{minmax3}

In the previous section, we claimed  that  $\mathcal M$ is neither linear nor uniformly elliptic, neither concave nor convex, except for $n=2$.  This is intuitive by the representation (\ref{infsup}):
$$\mathcal M(X) = \sup_{|\xi| =1}\inf_{|\eta|=1} \hbox{\text Tr}(X_{\xi,\eta}).$$

Nonetheless, we present a few counterexamples that support the above claim.
\vskip0.2cm

\begin{rem}
Let us consider the matrices $X_1={\bf e}_1\otimes{\bf e}_1-{\bf e}_3\otimes{\bf e}_3$, $X_2=- {\bf e}_1\otimes{\bf e}_1+{\bf e}_2\otimes{\bf e}_2$ and $X_3={\bf e}_1\otimes{\bf e}_1-{\bf e}_2\otimes{\bf e}_2-{\bf e}_3\otimes{\bf e}_3$.  Then $\lambda_1(X_i)=-1$ and $\lambda_3(X_i)=1$, so that $\mathcal M(X_i)=0$ for all $i=1,2,3$.

\begin{itemize}
 
\item[  ]

 \item[(i)]   The operator $\mathcal M$ is not linear in dimension $n\geq 3.$ In fact 
\begin{equation*}
\begin{split}
X_1-X_2&={\bf e}_1\otimes{\bf e}_1-{\bf e}_3\otimes{\bf e}_3 + {\bf e}_1\otimes{\bf e}_1-{\bf e}_2\otimes{\bf e}_2\\
&=2 {\bf e}_1\otimes{\bf e}_1 - {\bf e}_2\otimes{\bf e}_2 -{\bf e}_3\otimes{\bf e}_3
\end{split}
\end{equation*}
and therefore
$$\lambda_1(X_1-X_2)=-1,\quad \lambda_3(X_1-X_2)=2,$$
so that
 $$
 \mathcal M(X_1)-\mathcal M(X_2)=0 \neq 1= \mathcal M(X_1-X_2).
 $$

\item[(ii)]  The operator $\mathcal M$ is not uniformly elliptic in dimension $n\geq 3.$ 
In fact,  we note that $X_3 \le X_1$, and  
$$\mathcal M(X_3)=\mathcal M(X_1)=0,  \ \ \hbox{\text  but } X_3\neq X_1,$$
against the strictly increasing property on ordered chains observed in Subsection \ref{visco1} for the uniformly elliptic case.\\

 \item[(iii)] The operator $\mathcal M(X)$ is neither convex nor concave. In fact,  for every $t\in [0,1],$ it turns out that
\begin{equation*}
 \begin{split}
 tX_1+(1-t)X_2&=t({\bf e}_1\otimes{\bf e}_1-{\bf e}_3\otimes{\bf e}_3)\\
&+(1-t)(- {\bf e}_1\otimes{\bf e}_1+{\bf e}_2\otimes{\bf e}_2)\\
& =(2t-1){\bf e}_1\otimes{\bf e}_1+(1-t){\bf e}_2\otimes{\bf e}_2-t{\bf e}_3\otimes{\bf e}_3
 \end{split}
 \end{equation*}
 and therefore
 \begin{align*}
  \mathcal M(tX_1+(1-t)X_2)&=\lambda_1(tX_1+(1-t)X_2)+\lambda_3(tX_1+(1-t)X_2)\\
&=\min\{2t-1;-t\}+\max\{2t-1;1-t\},
 \end{align*}
so that $\mathcal M(tX_1+(1-t)X_2)=1-2t$ for $t \in(\frac13,\frac23).$ From this 
\begin{equation*}
\mathcal M(tX_1+(1-t)X_2)
\left\{
\begin{array}{ll}
 >0 & \hbox{\text for} \ t \in(\frac13,\frac12), \\ \\
 < 0 & \hbox{\text for} \ t \in(\frac12,\frac23),
\end{array}
\right.
\end{equation*}
while  it is plain that for every $t\in [0,1]$
 $$
 t\mathcal M(X_1)+(1-t)\mathcal M(X_2)=0.
 $$
 Thus $\mathcal M$ is neither convex nor concave.\qed
 \end{itemize}
 \end{rem}
\vskip0.2cm

Since $\mathcal M \in \mathcal A$ we already know that it is homogeneous of degree one (\ref{pos-hom}): for every $\rho\ge0$ and  for every $X\in \mathcal S^n$ 
\begin{equation}\label{hom}
\mathcal M(\rho X)=\rho\mathcal M(X).
\end{equation}

On the other hand, 
\begin{equation}\label{refl}
\mathcal M(- X)=\lambda_1(- X)+\lambda_n(-X)= -\lambda_n(X)-\lambda_1(X)=-\mathcal M(X),
\end{equation}
and therefore (\ref{hom}) continues to hold for $\rho<0$.

\vskip0.2cm

The next remark contains  a few comments on the representation (\ref{infsup}). 

\begin{rem}\label{infsup-orto-rem}
The operator $\mathcal M(X)$ can be put in the form
\begin{equation}\label{infsup-orto}
\mathcal M(X) = \sup_{|\xi| =1}\inf_{\substack{{|\eta|=1}\\{\eta\bot\xi}}} \hbox{\text Tr}(X_{\xi,\eta}).
\end{equation}
In order to prove this, we start observing that plainly
\begin{align*}
\mathcal M(X) &= \sup_{|\xi| =1}\inf_{|\eta|=1} \left(\langle X\xi,\xi\rangle +  \langle X\eta,\eta\rangle\right)\\
&\le \sup_{|\xi| =1}\inf_{\substack{{|\eta|=1}\\{\eta\bot\xi}}} \hbox{\text Tr}(X_{\xi,\eta}).
\end{align*}
To have also the reverse inequality, and so (\ref{infsup-orto}), we observe that the representative matrix $X_{\xi,\eta}$ of the quadratic form associated to $X$ restricted to $L(\xi,\eta)$, the subspace of $\mathbb R^n$ spanned by directions $\xi$ and $\eta$, has  trace 
$$\hbox{\rm Tr}(X_{\xi,\eta}) = \langle X\xi,\xi\rangle +  \langle X\eta,\eta\rangle,$$
and thus
$$\sup_{|\xi| =1}\inf_{\substack{{|\eta|=1}\\{\eta\bot\xi}}} \hbox{\text Tr}(X_{\xi,\eta})=\sup_{|\xi| =1} (\langle X\xi,\xi\rangle +\inf_{\substack{{|\eta|=1}\\{\eta\bot\xi}}}  \langle X\eta,\eta\rangle).$$
To  compute the inf in the latter equation, we may assume that $X$ is diagonal, by rotational invariance, with the eigenvalues $\lambda_1 \le \dots \le \lambda_n$ on the diagonal from the top to the bottom. Note also that in this case $\langle X\xi,\xi\rangle=\lambda_1\xi_1^2+\dots+\lambda_n\xi_n^2$ and $\langle X\eta,\eta\rangle=\lambda_1\eta_1^2+\dots+\lambda_n\eta_n^2$, so that by symmetry we may assume $\xi_i\ge0$ and $\eta_i\ge0$ for all $i=1,\dots,n$, that is 
$$\sup_{|\xi| =1}\inf_{\substack{{|\eta|=1}\\{\eta\bot\xi}}} \hbox{\text Tr}(X_{\xi,\eta})=\sup_{\substack{{|\xi|=1}\\{\xi \ge0}}} (\langle X\xi,\xi\rangle +\inf_{\substack{{|\eta|=1}\\{\eta \ge0}\\{\eta\bot\xi}}}  \langle X\eta,\eta\rangle).$$

\noindent Using the Lagrange multipliers $\lambda$ and $\mu$, the inf is obtained in correspondence of a critical point of the function
$$h(\eta,\lambda,\mu) :=  \langle X\eta,\eta\rangle - \lambda  (\langle \eta,\eta\rangle-1) - \mu  \langle \xi,\eta\rangle,$$
which solve the system
$$\left\{
\begin{array}{lll}
&X\eta=\lambda\eta+\frac\mu2\xi\\
&\langle \eta,\eta\rangle=1\\
& \langle \xi,\eta\rangle =0 
\end{array}\right.
$$
or equivalently
$$\left\{
\begin{array}{lll}
&\lambda_1\eta_1=\lambda\eta_1+\frac\mu2\xi_1\\
&\ \ \dots \ \ \  \dots \ \ \ \dots\\
&\lambda_n\eta_n=\lambda\eta_n+\frac\mu2\xi_n\\
&\eta_1^2+\dots+\eta_n^2=1\\
&  \xi_1\eta_1+\dots+\xi_n\eta_n =0 \,.
\end{array}\right.
$$
\vskip0.1cm
\noindent We can show that $\mu=0$. Otherwise, suppose by contradiction $\mu\neq0$. Let $I=\{ i \in\{1,\dots,n\} : \xi_i\neq 0\}$, which is non-empty because $|\xi|=1$. Then from above  $(\lambda_i-\lambda)\eta_i=\frac\mu2\xi_i\neq 0$, and so $\lambda\neq \lambda_i$ for all $i \in I$. Inserting $\eta_i =\frac\mu2\,\frac{\xi_i}{\lambda-\lambda_i}$ in the last row of the system we get 
$$\frac\mu2\sum_{ i\in I} \frac{\xi_i^2}{\lambda-\lambda_i}=0.$$
Since $\xi_i >0$ and $\eta_i> 0$ for $i \in I$, all the terms of the sum have the same sign (the sign of $\mu$), and this would imply $\mu=0$, against the assumption. Therefore critical points are not affected by the constraint $\eta\bot \xi$, and this proves the representation (\ref{infsup-orto}).
\end{rem}

\vskip0.2cm

If instead of ''sup inf''  as in  (\ref{infsup-orto}) we consider ''inf sup'', we re-obtain $\mathcal M$:
\begin{equation}\label{infsup-supinf}
\begin{split}
\mathcal M(X) &= \sup_{|\xi| =1}\inf_{\substack{{|\eta|=1}\\{\eta\bot\xi}}} \hbox{\text Tr}(X_{\xi,\eta})\\
&= \inf_{|\xi| =1}\sup_{\substack{{|\eta|=1}\\{\eta\bot\xi}}} \hbox{\text Tr}(X_{\xi,\eta})
\end{split}
\end{equation}
with or without the constraint $\eta\bot\xi$.

\vskip0.3cm

\subsection{Comparison with the partial trace operators operator $\mathcal P^{\pm}_k$}\label{compar4}

\vskip0.2cm

Let us give a comparative look to the partial trace operators (\ref{ptrace}): 
\begin{equation*}
\begin{split}
 \mathcal P_k^-(X)=\lambda_1(X)+\dots+\lambda_k(X), \ \ \mathcal P_k^+(X)=\lambda_{n-k+1}(X)+\dots+\lambda_n(X).
\end{split}
\end{equation*}

\begin{rem}\label{pk}

If  in  (\ref{infsup-supinf}) we consider ''sup sup'' or ''inf inf" instead od ''sup inf'' or ''inf sup '', it is not difficult to recognize, from (\ref{ptraceb}), that we obtain the above partial trace operators with $k=2$: 
\begin{equation}\label{infinf}
\begin{split}
&\mathcal P_2^-(X) = \inf_{|\xi| =1}\inf_{\substack{{|\eta| =1}\\{\eta\bot\xi}}} \hbox{\text Tr}(X_{\xi,\eta}), \\ 
& \mathcal P_2^+(X) = \sup_{|\xi| =1}\sup_{\substack{{|\eta| =1}\\{\eta\bot\xi}}} \hbox{\text Tr}(X_{\xi,\eta}).
\end{split}
\end{equation}

\vskip0.2cm

Next, we list some properties of operators $\mathcal P^\pm_k$. By definition, it is plain that  $\mathcal{P}^-_k\le \mathcal{P}^+_k$; in addition $\mathcal{P}^+_k$ and $\mathcal{P}^-_k$ are respectively subadditive and superadditive:
$$
 \mathcal{P}^-_k(X)+\mathcal{P}^-_k(Y) \le  \mathcal{P}^-_k(X+Y) \le \mathcal{P}^+_k(X+Y) \le \mathcal{P}^+_k(X)+\mathcal{P}^+_k(Y).
$$
Moreover, $\mathcal{P}^-_k(X)=-\mathcal{P}^+_k(-X)$, so that  from the left-hand inequality
$$
 \mathcal{P}^-_k(X+Y) \le   \mathcal{P}^-_k(X) - \mathcal{P}^-_k(-Y) = \mathcal{P}^-_k(X) + \mathcal{P}^+_k(Y)
$$
and from the right-hand
$$
 \mathcal{P}^+_k(X+Y) \ge   \mathcal{P}^+_k(X) - \mathcal{P}^+_k(-Y) = \mathcal{P}^+_k(X) + \mathcal{P}^-_k(Y)
$$

In particular, since $\lambda_1(X)= \mathcal{P}^-_1(X)$ and $\lambda_n(X)= \mathcal{P}^+_1(X)$,
\begin{equation}\label{l1}
\lambda_1(X)+\lambda_1(Y) \le \lambda_1(X+Y) \le  \lambda_1(X)+\lambda_n(Y)
\end{equation}
and
\begin{equation}\label{ln}
\lambda_1(X)+\lambda_n(Y) \le \lambda_n(X+Y) \le  \lambda_n(X)+\lambda_n(Y).
\end{equation}

We recall that the inequality stated above for the partial trace operators $\mathcal P^\pm_k$ continues to  hold for the Pucci extremal operators $\mathcal M^\pm_{\lambda,\Lambda}$, that can be in turn regarded as  Bellman operators. In fact, setting $\mathcal S^n_{\lambda,\Lambda}= \{A \in \mathcal S^n : \lambda I \le A \le \Lambda I\}$, where $I$ is the $n \times n$ identity matrix, we have
\begin{equation}\label{Pucci-B}
\mathcal M^+_{\lambda,\Lambda}(X) = \sup_{A \in \mathcal S^n_{\lambda,\Lambda}}{\text {Tr}}(AX), \ \ \mathcal M^-_{\lambda,\Lambda}(X) = \inf_{A \in \mathcal S^n_{\lambda,\Lambda}}{\text {Tr}}(AX).
\end{equation}
\end{rem}

\vskip0.3cm

\subsection{Duality}\label{dua5}

\vskip0.2cm

 Let $\mathcal F$ be a fully nonlinear degenerate elliptic operator. If $\mathcal F$ is linear and  $u$ is a subsolution of the equation $\mathcal F(D^2u)=f$, then $v=-u$  is a supersolution of the equation $\mathcal F(D^2v)=-f$. 

If we deal with an arbitrary fully nonlinear operator and  $u$  is a subsolution to $F(D^2u)= f$, then $v=-u$ is a supersolution of an equation $\tilde{\mathcal F}(D^2v)=-f$ for the dual operator $\tilde{\mathcal F}$,
\begin{equation}\label{dual}
\tilde{\mathcal F}(X)=-{\mathcal F}(-X),
\end{equation}
which is  is in general different from $\mathcal F$. Moreover, $\overline{\mathcal F}$ is degnerate (uniformly) elliptic if $\mathcal F$ is degenerate (uniformly) elliptic.

Computing the dual of the operators introduced above, we note that by homogeneity for the min-max operator $\mathcal M$ we have $\tilde{\mathcal M} = \mathcal M$ as in the case of linear operators, while the upper and lower partial trace operators  are each one the dual of the other one, $\tilde{\mathcal P}^\pm_k = \mathcal P^\pm_k$, as well as the maximal and the minimal the Pucci operators, $\tilde{\mathcal M}_{\lambda,\Lambda}^\pm =  \mathcal M_{\lambda,\Lambda}^\mp$. In the general case  $\mathcal M_{\bf a}\in \overline{\mathcal A}$,  we have $\tilde{\mathcal M}_{\bf a} = \mathcal M_{{\bf a}'}$, where ${\bf a}'=(a_n,a_{n-1},\dots,a_1)$,.

\section{Auxiliary results}\label{aux}

In this section we apply the Perron method, well known in the literature, see for instance \cite{USE} and \cite{HL1}, in order to show: weak maximum and  comparison principles, existence and uniqueness of solutions, see respectively Subsections \ref{weakcompar}, \ref{existuniq}. The proofs are based on the properties of our operators, suitably exploited, and an appropriate adaptation of arguments used for  the uniformly elliptic case. In Subsection \ref{Radial_solutions} we obtain the radial representation of the operators $\mathcal{M}_{\bf a}.$
\vskip0.2cm

\subsection{Weak maximum and comparison principles}\label{weakcompar}

\vskip0.2cm

The following comparison principle holds between viscosity subsolutions and supersolutions of the equation $\mathcal M_{\bf a}[u]=f$ in a bounded domain $\Omega$, as proved for uniformly elliptic operators in the basic paper of Crandall-Ishii-Lions \cite{USE}.

 \begin{thm}\label{comparison} {\rm(comparison principle)} Let $u \in usc(\overline \Omega)$ and $v \in  lsc(\overline \Omega)$, such that $\mathcal M_{\bf a}(D^2u)\ge f$ and $\mathcal M_{\bf a}(D^2v)\le f$ in $\Omega$ are satisfied  in the viscosity sense, respectively, where $\Omega$  is a bounded  open set of $\mathbb{R}^n$, $\mathcal M_{\bf a} \in \overline{\mathcal A}$ and $f$ is a bounded continuous function in $\Omega$.  If $u \le v$ on $\partial \Omega$, then $ u\le v$ in $\Omega$.
\end{thm} 

Letting $v\equiv0$ and $f\equiv0$, we obtain  the following weak maximum principle.

\begin{cor}\label{wMP} {\rm (weak maximum principle)}
Let $u\in usc(\overline{\Omega}),$ where $\Omega$ is a bounded domain of $\mathbb R^n$. If $\mathcal M_{\bf a}(D^2u(x))\geq 0$ in $\Omega$ in the viscosity sense for some $\mathcal M_{\bf a} \in \overline{\mathcal A}$, then
$$
\max_{\overline{\Omega}} u=\max_{\partial \Omega} u. 
$$ 
On the other hand, if $u\in lsc(\overline{\Omega})$ is a viscosity solution of the differential inequality $\mathcal M_{\bf a}(D^2u(x))\leq 0$ in $\Omega$ for some $\mathcal M_{\bf a} \in \overline{\mathcal A}$, then
$$
\min_{\overline{\Omega}} u=\min_{\partial \Omega} u. 
$$ 
\end{cor}
\vskip0.2cm

{\it Proof of Theorem} \ref{comparison}. The case of $f(x)\equiv 0$ is covered in \cite[Theorem 6.5]{HL1}. 

In fact,  considering the Dirichlet set $F=\{X \in \mathcal S^n : \mathcal M_{\bf a}(X) \ge 0\}$ and its dual set $\tilde F=\{X \in \mathcal S^n : \mathcal M_{{\bf a}'}(X) \ge 0\}$ in the geometric setting of Harvey-Lawson \cite{HL1}, then by  our assumptions $u,-v \in$\,usc$(\overline\Omega)$ are of type $F$ and $\tilde F$ in $\Omega$, and our comparison principle is deduced the subaffinity of $u-v$ established there.

\vskip0.2cm 

For sake of completeness, we give an analytic proof based on the device contained in the proof of  \cite[Theorem 3.3]{USE} by Crandall-Ishii-Lions. See also \cite{BM}.
\vskip0.2cm

We have to show that, under the given assumptions, the maximum of $u-v$ must be realized on $\partial\Omega$. 

\vskip0.2cm

i) Firstly,  setting $u_\varepsilon(x) = u(x)+\frac12\varepsilon|x|^2$, we prove that $u_\varepsilon-v$ cannot have a positive maximum in $\Omega$, for all fixed $\varepsilon>0$. 

Actually,
\begin{equation}\label{u-epsilon}
\begin{split}
\mathcal M_{\bf a}(D^2u_\varepsilon):=&\,\sum_{i=1}^na_i\lambda_i(D^2u_\varepsilon)=\sum_{i=1}^na_i\lambda_i(D^2u+\varepsilon I) \\
\ge &\, f(x)+|{\bf a}|\,\varepsilon\,;\\
\mathcal M_{\bf a}(D^2v):=&\,\sum_{i=1}^na_i\lambda_i(D^2v) \le \,f(x),
\end{split}
\end{equation}
where $|{\bf a}|=a_1+\dots+a_n>0$.

Supposing, by contradiction, that $u_\varepsilon-v$ has a positive maximum in $\Omega$ and following the proof of  \cite[Theorem 3.3 ]{USE}, for all $\alpha>0$ there exist points $x_\alpha, y_\alpha \in \Omega$ and matrices $X_\alpha, Y_\alpha \in \mathcal S^n$, such that
\begin{equation}\label{Hessian}
-3\alpha\left(
\begin{array}{ll}
I & 0\\
0 & I
\end{array}\right) \le \left(
\begin{array}{ll}
\hskip-0.1cm X_\alpha & 0\\
0 & \hskip-0.2cm-Y_\alpha
\end{array}\right) \le
3\alpha\left(
\begin{array}{cc}
I & -I\\
\hskip-0.2cm-I &\hskip0.2cm I
\end{array}\right).
\end{equation}
and  
\begin{equation}\label{eq}
\sum_{i=1}^na_i\lambda_i(X_\alpha) \ge f(x_\alpha)+|{\bf a}|\,\varepsilon, \ \  \sum_{i=1}^na_i\lambda_i(Y_\alpha) \le f(y_\alpha).
\end{equation}
Moreover
\begin{equation}\label{limit}
\lim_{\alpha\to \infty}\alpha|x_\alpha-y_\alpha|^2=0.
\end{equation}

Noting that (\ref{Hessian}) implies $X_\alpha \le Y_\alpha$, from (\ref{eq}) we get 
 $$f(x_\alpha)+|{\bf a}|\,\varepsilon \le\sum_{i=1}^na_i\lambda_i(X_\alpha) \le \sum_{i=1}^na_i\lambda_i(Y_\alpha) \le f(y_\alpha).$$
Taking the limit as $\alpha \to \infty$ and using  (\ref{limit}), by the continuity of $f(x)$ we have a contradiction: $\varepsilon \le0$. Therefore $u_\varepsilon-v$ cannot have a positive maximum in $\Omega$.

\vskip0.2cm
ii) From i) it follows, for all $\varepsilon >0$, that $ \max_{\overline \Omega}(u_\varepsilon-v) \le  \max_{\partial \Omega}(u_\varepsilon-v)$. Taking into account that $u \le v$ on $\partial \Omega$, then we have
$$u(x)+\frac12\,\varepsilon\,|x|^2 - v(x) \le \frac12\,\varepsilon\,R^2, \ \ \text{for } x \in \Omega,$$
where $R>0$ is the radius of a ball $B_R$ centered at the origin such that $\Omega \subset B_R$. \\ 
Letting $\varepsilon \to 0^+$, we conclude that $u\le v$ in $\Omega$, as claimed. \qed

\vskip0.2cm

From Corollary \ref{wMP} we deduce the following uniform estimates for viscosity solutions of the equation $\mathcal M_{\bf a}[u]=f$ in  a bounded domain $\Omega$

\begin{prop}\label{unif-est-cor} {\rm (uniform estimate)}
Let $u\in usc(\overline{\Omega}),$ where $\Omega$ is a bounded domain of $\mathbb R^n$. If $\mathcal M_{\bf a}(D^2u(x))\geq f(x)$ in $\Omega$ in the viscosity sense for some $\mathcal M_{\bf a} \in \overline{\mathcal A}$ and $f$ is bounded below in $\Omega$, then
$$
u(x)\le\max_{\partial \Omega} u^ + Cd^2\|f^-\|_{_{L^\infty(\Omega)}} \quad \forall\,x \in \overline \Omega,
$$
where $C$ is a positive constant, which can be chosen equal to $1/|{\bf a}|$.

On the other hand, if $u\in lsc(\overline{\Omega})$ is a viscosity solution of the differential inequality $\mathcal M_{\bf a}(D^2u(x))\leq f(x)$ in $\Omega$ for some $\mathcal M_{\bf a} \in \overline{\mathcal A}$  and $f$ bounded above in $\Omega$, then
$$
u(x)\ge \min_{\partial \Omega} u - Cd^2\|f^+\|_{_{L^\infty(\Omega)}} \quad \forall\,x \in \overline \Omega. 
$$ 
\end{prop}
\begin{proof} Let us prove the first one. Setting $K^-=\|f^-\|_{L^\infty(\Omega)}$, the function $v=u+\frac{K^- }{2|{\bf a}|}|x|^2$ is a subsolution of the equation $\mathcal M_{\bf a}[v]=0$. By Corollary \ref{wMP} we get $v(x) \le \max_{\overline \Omega}v$, so that
$$
u(x) \le v(x) \le \max_{\partial \Omega} u +\frac{K^- }{2|{\bf a}|}d^2,
$$
which yields the first inequality of the estatement.
\end{proof}

\vskip0.2cm

\subsection{Existence and uniqueness}\label{existuniq}
\vskip0.2cm

As a consequence of the above comparison principle, we can also prove an existence and uniqueness result for the Dirichlet problem
 in bounded domains $\Omega$ via the Perron method, assuming that $\Omega$ has a uniform exterior cone condition, see \cite{CCKS}: there exist $\theta_0 \in (0,\pi)$ and $r_0>0$ so that for every $y \in \partial\Omega$ there is a rotation $\mathcal R = \mathcal R(y)$ such that 
\begin{equation}\label{exterior}
\overline\Omega \cap B_{r_0}(y)\subset y + \mathcal R\Sigma_{\theta_0},
\end{equation}
where
\begin{equation}\label{externalcone}
\Sigma_{\theta_0} = \{x \in \mathbb R^n : x_n \ge |x|\cos{\theta_0}\}.
\end{equation}

\begin{thm}\label{exist-uniq} Let $\Omega$ be a bounded domain of $\mathbb R^n$ endowed with a uniform exterior cone condition. Let $g$ be a continuous function on the boundary $\partial \Omega$, and $f$ be a continuous and bounded function in $\Omega$. Then for  $\mathcal M_{\bf a} \in \mathcal A$ the Dirichlet problem
\begin{equation}\label{DP}
\left\{
\begin{array}{ll}
\mathcal M_{\bf a}(D^2u) = f & \text{\rm in} \ \Omega\\
u=g & \text{\rm on} \ \partial\Omega
\end{array}\right.
\end{equation}
 has a unique viscosity solution $u \in C(\overline\Omega)$.
\end{thm}

\begin{proof} According to the Perron method \cite[Theorem 4.1]{USE}, we need a comparison principle, and the existence of a subsolution and a supersolution of the equation $\mathcal M_{\bf a}(D^2u)=f$.

Since the comparison principle holds by Theorem \ref{comparison}, we only need to look for a viscosity subsolution $\underline u\in$\,usc$(\overline \Omega)$  and a viscosity supersolution  $\overline u\in$\,lsc$(\overline \Omega)$ of the equation $\mathcal M_{\bf a}(D^2u)=f(x)$ such that $\underline u=g=\overline u$ on $\partial \Omega$.

To do this, we will use the following inequalities, see (\ref{Ma-elliptic}) and (\ref{Ma-lip}):

\begin{equation}\label{Pucci-up}
\begin{split}
\mathcal M_{\bf a}(X) &= a_1\lambda_1(X)+\dots+a_n\lambda_n(X)\\
&= n\frac{a_1}{n}\lambda_1(X)+\dots+a_n\lambda_n(X)\\
&\leq \frac{a_1}{n}\lambda_1(X)+\sum_{i=2}^{n}\left(\frac{a_1}{n}+a_i \right)\lambda_i(X)\\
&=: \mathcal M_{\overline{\bf a}}(X) 
\end{split}
\end{equation}
and
\begin{equation}\label{Pucci-down}
\begin{split}
\mathcal M_{\bf a}(X) &= a_1\lambda_1(X)+\dots+a_n\lambda_n(X)\\
&= a_1\lambda_1(X)+\dots+n\frac{a_n}{n}\lambda_n(X)\\
&\geq \sum_{i=1}^{n-1}\left(a_i+\frac{a_n}{n}\right)\lambda_i(X)+\frac{a_n}{n}\lambda_n(X) \\
&=: \mathcal M_{\underline{\bf a}}(X).
\end{split}
\end{equation}

If $\mathcal M_{\bf a} \in \mathcal A_1$, then  $\mathcal M_{\overline{\bf a}}$ is uniformly elliptic with  ellipticity constants
\begin{equation}\label{elliptic-up} 
\begin{split}
\overline\lambda= \frac{a_1}{n},  \ \ \overline\Lambda= \frac{a_1}{n}+\max_{2 \le i \le n}a_i 
\end{split}
\end{equation}
so that
\begin{equation}\label{Ma-Pucci+asy}
\mathcal M_{\bf a}(X) \le \mathcal M_{\overline{\bf a}}(X) \le \mathcal M_{\frac{a_1}{n},|{\bf a}|}^+
\end{equation}
and, if  $\mathcal M_{\bf a} \in \mathcal A_n$, then $\mathcal M_{\underline{\bf a}}$ is uniformly elliptic with ellipticity constants
\begin{equation}\label{elliptic-down} 
\begin{split}
\underline\lambda=\frac{a_n}{n},  \ \ \underline\Lambda=  \frac{a_n}{n}+\max_{1 \le i \le n-1}a_i,
\end{split}
\end{equation}
so that
\begin{equation}\label{Ma-Pucci-asy}
\mathcal M_{\bf a}(X) \ge \mathcal M_{\underline{\bf a}}(X) \ge \mathcal M_{\frac{a_n}{n},|{\bf a}|}^-(X).
\end{equation}

Therefore, if  $\mathcal M_{\bf a} \in \mathcal A$, and $\lambda^*$ and $\Lambda^*$ are positive numbers such that
\begin{equation}\label{elliptic-const}
\begin{split}
&\lambda^* \le \min(\underline \lambda, \overline \lambda)=\frac{a^*}{n}\equiv \frac{\min(a_1,a_n)}{n} \\
&\Lambda^* \ge \max(\underline \Lambda, \overline \Lambda)\ \ge  |{\bf a}| \equiv a_1+\dots+a_n,
\end{split}
\end{equation}
by the extremality properties (\ref{ue0}) of Pucci operators, from (\ref{Ma-Pucci+asy}) and (\ref{Ma-Pucci-asy}) we have
\begin{equation}\label{Ma-Pucci}
\mathcal M^-_{\frac{a^*}{n},|{\bf a}|}(X) \le \mathcal M_{\bf a}(X) \le \mathcal M^+_{\frac{a^*}{n},|{\bf a}|}(X).
\end{equation}

Next, setting $K=\sup_\Omega |f|$, we solve by \cite[Proposition 3.2]{CCKS} the Dirichlet problems
\begin{equation}\label{Dir-sub}
\left\{
\begin{array}{ll}
\mathcal M^-_{\frac{a^*}{n},|{\bf a}|}(D^2\underline u)=K & \hbox{\text in} \ \Omega\\
\underline u = g & \hbox{\text on} \ \partial\Omega
\end{array}
\right.
\end{equation}
and
\begin{equation}\label{Dir-super}
\left\{
\begin{array}{ll}
\mathcal M^+_{\frac{a^*}{n},|{\bf a}|}(D^2\overline u)=-K & \hbox{\text in} \ \Omega\\
\overline u = g & \hbox{\text on} \ \partial\Omega
\end{array}
\right.
\end{equation}
Since obviously $-K \le f(x) \le K$ for all $x \in \Omega$, from (\ref{Ma-Pucci}) it follows that $\underline u$ and $\overline u$ provide a subsolution and a supersolution that we were searching for, concluding the proof.
\end{proof}

An existence and uniqueness result is provided  for all the class $\mathcal A$ by \cite[Theorem 6.2]{HL1} for smooth boundaries.

A weaker condition can be obtained from  \cite{BR}, where the authors consider in detail the case ${\bf a}={\bf e}_j$, namely the equation $\lambda_j[u]=0$, and prove an existence and uniqueness theorem for the Dirichlet problem (\ref{DP}) with a sharp geometric condition on the boundary of $\Omega$, depending on $j$. 

From there, we take a sufficient condition to solve the Dirichlet problem for any equation  $\lambda_j(D^2u)=0$, $j=1,\dots,n$: given  $y \in \partial \Omega$,  for every $r > 0$ there exists $\delta> 0$ such that, for every $x\in B_\delta(y)$ and direction $v \in \mathbb R^n$ ($|v|=1$),
$$(x+\mathbb R\,v)\cap B_r(y) \cap \partial\Omega \neq \emptyset. \eqno(G_1)$$

This condition does not require smooth boundary, but it is nevertheless stronger than the exterior cone property.

\begin{thm}\label{exist-uniq'} Let $\Omega$ be a bounded domain of $\mathbb R^n$ satisfying condition $(G_1)$. Let $g$ be a continuous function on the boundary $\partial \Omega$, and $f$ be a continuous and bounded function in $\Omega$. Then for  $\mathcal M_{\bf a} \in \overline{\mathcal A}$ the Dirichlet problem
\begin{equation}\label{DP}
\left\{
\begin{array}{ll}
\mathcal M_{\bf a}(D^2u) = f & \text{\rm in} \ \Omega\\
u=g & \text{\rm on} \ \partial\Omega
\end{array}\right.
\end{equation}
 has a unique viscosity solution $u \in C(\overline\Omega)$.
\end{thm}

\begin{proof} Following the same lines of the proof of Theorem \ref{exist-uniq}, we only need to look for a viscosity subsolution $\underline u\in$\,usc$(\overline \Omega)$  and a viscosity supersolution  $\overline u\in$\,lsc$(\overline \Omega)$ of the equation $\mathcal M_{\bf a}(D^2u)=f(x)$ such that $\underline u=g=\overline u$ on $\partial \Omega$.

To do this, we observe this time:
\begin{equation}\label{lambdan-up}
\mathcal M_{\bf a}(X) = a_1\lambda_1(X)+\dots+a_n\lambda_n(X) \le |{\bf a }| \lambda_n(X)
\end{equation}
and
\begin{equation}\label{lambda1-down}
\mathcal M_{\bf a}(X) = a_1\lambda_1(X)+\dots+a_n\lambda_n(X) \ge |{\bf a }| \lambda_1(X)\,.
\end{equation}

\vskip0.2cm
Next, setting $K=\sup_\Omega |f|$, we solve by \cite[Theorem 1]{BR} the Dirichlet problems
\begin{equation}\label{Dir-sub}
\left\{
\begin{array}{ll}
 |{\bf a }|\lambda_1(D^2\underline u)=K & \hbox{\text in} \ \Omega\\
\underline u = g & \hbox{\text on} \ \partial\Omega
\end{array}
\right.
\end{equation}
and
\begin{equation}\label{Dir-super}
\left\{
\begin{array}{ll}
|{\bf a }|\lambda_n(D^2\overline u)=-K & \hbox{\text in} \ \Omega\\
\overline u = g & \hbox{\text on} \ \partial\Omega
\end{array}
\right.
\end{equation}
As in the proof of Theorem \ref{exist-uniq}, $\underline u$ and $\overline u$ provide a subsolution and a supersolution, concluding the proof.
\end{proof}

\vskip0.3cm

\subsection{Radial solutions}\label{Radial_solutions}

\vskip0.2cm

We compute $\mathcal M$ on radial functions  $u(x)=v(|x|)$. Suppose   $v$ is $C^2,$ we recall that for $x \neq 0$:
 \begin{equation*}
\begin{split}
&Du(x) = v'(|x|)\,\frac{x}{|x|}\\
& D^2u(x)=v''(|x|)\,\tfrac{x}{|x|}\otimes \tfrac{x}{|x|}+\tfrac{v'(|x|)}{|x|}(I-\tfrac{x}{|x|}\otimes \tfrac{x}{|x|}),
\end{split}
 \end{equation*}
  where $\frac{x}{|x|}\otimes \frac{x}{|x|}\geq 0,$  $I-\frac{x}{|x|}\otimes \frac{x}{|x|}\geq 0$ and
\begin{equation*}
\begin{split}
& \left\langle \tfrac{x}{|x|}\otimes \tfrac{x}{|x|}h,h\right\rangle=\left\langle \tfrac{x}{|x|},h\right\rangle^2,\\
& \left\langle (I-\tfrac{x}{|x|}\otimes \tfrac{x}{|x|})h,h\right\rangle=|h|^2-\left\langle \tfrac{x}{|x|},h\right\rangle^2.\end{split}
 \end{equation*}
As a  consequence, $\frac{x}{|x|}$ is eigenvector of $\frac{x}{|x|}\otimes \frac{x}{|x|}$ with eigenvalue $1$, and  of  $I-\frac{x}{|x|}\otimes \frac{x}{|x|}$ with eigenvalue $0$. Conversely, all non-zero vectors orthogonal to $\frac{x}{|x|}$ are eigenvectors of $\frac{x}{|x|}\otimes \frac{x}{|x|}$ with  eigenvalue $0$ and of $I-\frac{x}{|x|}\otimes \frac{x}{|x|}$ with eigenvalue $1$. 

 It follows that
  $$
  \lambda_1(D^2u(x))+\lambda_n(D^2u(x))=v''(|x|)+\frac{v'(|x|)}{|x|}.
  $$

From this we deduce useful properties which are collected in the following remark. 

\begin{rem}\label{radial}
\begin{itemize}

\item[  ]
 
 \item[(i)]   The operator $\mathcal M$ is linear on the radial functions $u(x)=v(|x|)$.
\item[(ii)]  Any function of the form
$$\varphi(x)=a+b\log |x|,$$
with $a$ and $b$ constant,  is a solution of $\mathcal M[u]=0$ in $\mathbb{R}^n\setminus\{0\}.$

 \item[(iii)]   Recall that the $k-$th Hessian operator, $ k=1,\dots,n$, for radial functions is:
  $$
S_k(D^2u)=\binom{n-1}{k-1}\left(\frac{v'}{|x|}\right)^{k-1}\left(v''+\frac{n-k}{k}\frac{v'}{|x|}\right).
$$ 
In case $n=2k$ the radial  solutions of the equation $S_{\frac{n}{2}}(D^2u)=0$  are just  the radial solutions of $\mathcal M(D^2u)=0$.\qed
 \end{itemize}
\end{rem}

Recalling that $|{\bf a}|=a_1+\dots+a_n$, let $\hat a_j = |{\bf a}|-a_j$, $j=1,\dots,n$.
 More generally, for $\mathcal M_{\bf a} \in \overline{\mathcal A}$ the non-constant radial solutions in $\mathbb R^n\backslash\{0\}$, up to a multiplicative constant, are 
\begin{equation}\label{radialsoln}
\varphi(x)=\left\{
\begin{array}{ll}
|x|^{-\gamma_n} & \hbox{\text if} \  \hat a_n>a_n \vspace{0.1cm}\\ 
\log|x|^{-1} & \hbox{\text if} \  \hat a_n= a_n \ \vspace{0.1cm}\\ 
|x|^{\gamma_1} & \hbox{\text if} \  a_1>\hat a_1 \vspace{0.1cm}\\ 
\log|x| & \hbox{\text if} \  a_1=\hat a_1
\end{array}
\right.
\end{equation}
where  $\gamma_n =\displaystyle \frac{\hat a_n}{a_n}-1$.


\section{The ABP estimate}\label{abp}

The celebrated ABP estimate provides a uniform estimate for the solution of an elliptic equation $F[u]=f$ with the $L^n$-norm of $f$. The original inequality, for linear uniformly elliptic operators in bounded domains, goes back to Alexandroff \cite{ALE,ALE2}, but it already appears in Bakel'man \cite{BAK}. A different version has been later obtained by Pucci \cite{PUC}. 

In \cite{CAB} it was also proved for the first time an ABP estimate for solutions in $W^{2,p}_{\text{loc}}(\Omega)$ of the equation $F[u]=f$ with $f \in L^p$ and $p \in (n/2,n)$. A result of this kind is known in the framework of $L^p$-viscosity solutions \cite{CCKS} as the generalized maximum principle, which can be found in \cite{FOK} and \cite{CS} in the fully nonlinear uniformly elliptic case.

It is worth noticing that an ABP estimate for degenerate elliptic equations of $p$-Laplacian type has been proved by Imbert \cite{IMB}.

An extension of this inequality to unbounded domains $\Omega$ of cylindrical type for bounded solutions in $W^{2,n}_{\text{loc}}(\Omega)$ is due to Cabr\'e \cite{CAB}. By domains of cylindrical type we intend here a measure-geometric condition, which is satisfied by cylinders and goes back to a famous paper of Berestycki-Nirenberg-Vardhan \cite{BNV}, containing a characterization of the weak maximum principle. In subsequent papers the results of \cite{CAB} have been generalized to domains of conical type \cite{CV, VIT,VIA} and to viscosity solutions of  fully nonlinear uniformly elliptic equations \cite{CDLV}, and then to different classes of degenerate elliptic equations \cite{BCDV,CDV2,CDV3}.
\vskip0.2cm

\vskip0.2cm

The proof of the ABP estimates of Theorem \ref{ABP-two} is based on the geometrical argument used in the proof of Theorem 9.1 of the Gilbarg-Trudinger's book \cite{GT} for classical solutions.

We denote by  $\Gamma^+_u$ the upper convex envelope of $u$, the smallest concave function greater than $u$  in $\Omega$, and by $\Gamma^-_u$ the lower convex envelope of $u$, the largest convex function smaller than $u$  in $\Omega$.
\vskip0.2cm

\begin{lem}\label{ABP-classic}
Let $\Omega$ be a bounded domain with diameter $d$, and $\mathcal M_{\bf a} \in \mathcal A_1$. For every $u\in C^2(\Omega)\cap C^0(\bar{\Omega})$ such that $u \le 0$ on $\partial \Omega$  we have
\begin{equation}\label{ABP-classic-sub}
\sup_{\Omega}u^+ \le \frac{1}{a_1}\frac{d}{\omega_n^{1/n}} \left\|\mathcal M_{\bf a}(D^2 u(x))^-\right\|_{L^n(\{\Gamma_u^+=u\}}\,,
\end{equation}
where $\omega_n$ denotes the Lebesgue measure of the $n$-dimensional unit ball. 

On the other hand, let us assume $\mathcal M_{\bf a} \in  \mathcal A_n$. For every $u\in C^2(\Omega)\cap C^0(\bar{\Omega})$ such that $u \ge 0$ on $\partial \Omega$   we have
\begin{equation}\label{ABP-classic-super}
\sup_{\Omega}u^- \le \frac{1}{a_n}\frac{d}{\omega_n^{1/n}} \left\|\mathcal M_{\bf a}(D^2 u(x))^-\right\|_{L^n(\{\Gamma_u^-=u\})}\,,
\end{equation}
\end{lem}

\begin{proof}
Let us prove the first estimate (\ref{ABP-classic-sub}). We argue following the proof of Lemma 9.2 in \cite{GT} and of Lemma 3.4 in \cite{CC}, denoting by $\chi_u:\Omega \to \mathbb R^n$ the normal mapping 
\begin{equation}\label{normal-map}
\chi_u(z)=\{p\in \mathbb{R}^n:\ u(x)\leq u(y)+\langle p,x-z\rangle\ \,\forall\, x\in \Omega\}, \ \  z \in \Omega\,.
\end{equation}

We e remark that on the upper contact set $\{\Gamma^+_u=u\}$ the eigenvalues of $D^2u$ are non-positive, and
the Lebesgue measure of $\chi_u$  can be estimated as
\begin{equation}\label{normal-measure}
|\chi_u(\Omega)|\leq \int_{\Gamma^+_u=u}|\mbox{det}\, D^2 u(x)|\,dx.
\end{equation}

If $u \le 0$ in $\Omega$, then inequality (\ref{ABP-classic-sub}) is obvious.  Suppose then $u$ realizes a positive maximum at a point $y\in \Omega$, and recall that $\Omega \subset B_d(y)$. 

Let $\kappa$  be the function whose graph is the cone $K$ with vertex $(y,u(y))$ and base $\partial B_d(y)$, then $  \chi_\kappa(\Omega)\subset \chi_u(\Omega) $. Then  $\chi_u(\Omega) $ and contains all the slopes of $B_{u(y)/d}$, so that $\omega_n \,(u(y)/d)^n \le |\chi_u(\Omega)|$  and by (\ref{normal-measure})
\begin{equation}\label{maine}
u^+(y) \leq \frac{d}{\omega_n^{1/n}}\left(\int_{\Gamma^+_u=u}|\mbox{det} D^2 u(x)|\,dx\right)^{1/n}.
\end{equation}
Since on the contact set we have $|\lambda_n|\leq |\lambda_{n-1}|\leq \dots\leq |\lambda_1|$, so that
\begin{equation}\label{det}
\begin{split}
|\mbox{det} D^2 u| &=|\lambda_1(D^2u)|\cdots|\lambda_n(D^2u)| \leq |\lambda_1(D^2u)|^{n}\\
&= \frac1{a_1^n}\, \left|a_1\lambda_1(D^2u)\right|^{n}\leq \frac1{a_1^n}\, \left|\sum_{i=1}^n a_i\lambda_i(D^2u)\right|^{n}\\
&=\frac1{a_1^n}\, \left((\mathcal M_{{\bf a}}(D^2u))^-\right)^n.
\end{split}
\end{equation}
From (\ref{maine}) and (\ref{det}) we obtain  the estimate from above (\ref{ABP-classic-sub}).

\vskip0.2cm

For the estimate from below, we can apply  (\ref{ABP-classic-sub}) with $v=-u$ instead of $u$, observing that by assumption $v \le 0$ on $\partial \Omega$ and by duality 
\begin{equation}\label{dual-super}
\mathcal M_{{\bf a}'}(D^2v)) = -\mathcal M_{{\bf a}}(D^2u)).
\end{equation}

Then
\begin{equation}\label{ABP-classic-presuper}
\begin{split}
\sup_{\Omega}u^- &=\sup_\Omega v^+\\
&\le \frac{1}{a'_1}\frac{d}{\omega_n^{1/n}} \left\|\mathcal M_{{\bf a}'}(D^2 v(x))^+\right\|_{L^n(\{\Gamma_v^+=v\})}\\
&= \frac{1}{a_n}\frac{d}{\omega_n^{1/n}} \left\|\mathcal M_{{\bf a}}(D^2 u(x))^-\right\|_{L^n(\{\Gamma_u^-=u\})}\,.
\end{split}
\end{equation}

\end{proof}

Theorem \ref{ABP-two} is obtained combining the two unilateral ABP estimates, which hold separately for subsolutions and supersolutions, contained in the following result.

\begin{thm}\label{ABP}
Let $\Omega$ be a bounded domain of diameter $d$. Let  $f$ be continuous and bounded in $\Omega$. There exist an universal constant $C_n>0$, depending only on $n$ 

(i) for viscosity subsolutions $u\in$\,usc$(\Omega)$   of the equation $\mathcal M_{\bf a}[u]=f$   in $\Omega$ with $\mathcal M_{\bf a} \in \mathcal A_1$ 
\begin{equation}\label{ABP-est-sub}
\sup_{\Omega}u^+\leq \sup_{\partial\Omega}u^+ + \frac{C_n}{a_1}\,d\,\|f\|_{L^n(\Omega)};
\end{equation}

(ii)  for viscosity  supersolutions $u\in$\,lsc$(\Omega)$   of the equation $\mathcal M_{\bf a}[u]=f$  in $\Omega$ with $\mathcal M_{\bf a} \in \mathcal A_n$ and
\begin{equation}\label{ABP-est-super}
\sup_{\Omega}u^-\leq \sup_{\partial\Omega}u^- + \frac{C_n}{a_n}\,d\,\|f\|_{L^n(\Omega)}.
\end{equation}
\end{thm}

For classical solutions the proof follows directly from Lemma \ref{ABP-classic}.
\vskip0.2cm

{\bf Proof of Theorem \ref{ABP}: classical solutions.}  For subsolutions, supposing $\mathcal M_{\bf a}(D^2u) \ge f \ge -f^-$, we have $\mathcal M_{\bf a}(D^2u)^- \le f^-$. From Lemma \ref{ABP-classic}, passing to $u-\sup_{\partial \Omega}u$ in (\ref{ABP-classic-sub}), we get inequality (\ref{ABP-est-sub}). For supersolutions, supposing $\mathcal M_{\bf a}(D^2u) \le f \le -f^+$, we have $\mathcal M_{\bf a}(D^2u)^+ \le f^+$ . From Lemma \ref{ABP-classic}, passing to  $u-\inf_{\partial \Omega}u$ in (\ref{ABP-classic-super}), we get inequality (\ref{ABP-est-super}).\qed

\vskip0.2cm

To consider viscosity subsolutions, we extend $u^+=\max(u,0)$ and $f^-$ to zero outside $\Omega$, keeping the respective notations, and observing that in the viscosity setting $\mathcal M_{\bf a}(D^2u^+)\ge -f^-$ in $\mathbb R^n$. For viscosity supersolutions we extend $u^-=\max(u,0)$ and $f^+$ to zero outside $\Omega$ so that $\mathcal M_{\bf a}(D^2u^-)\le f^+$ in $\mathbb R^n$.

\vskip0.2cm 
In what follows we will refer to $\Gamma^+_u$ and $\Gamma^-_u$ as to the upper and the lower convex envelope of $u^+$ and $-u^-$, respectively, relative to the ball $B_{2d}$ concentric with a ball $B_d$ of radius $d$ containing $\Omega$. 

\vskip0.2cm
The key tool is the following lemma, which allows to apply the classical ABP estimates obtained before to viscosity subsolutions and supersolutions and is the counterpart of Lemma 3.3 of \cite{CC}.

\begin{lem} \label{C1,1} Let $\mathcal M_{\bf a} \in \mathcal A$. Let $u\in$\,{\rm lsc}$(\overline B_\delta)$, where $ B_\delta=\{|x-x_0|<\delta\}$, such that
\begin{equation}\label{D2-above}
\mathcal M_{\bf a}(D^2u) \le f \ \ {\rm  in} \ B_\delta
\end{equation}
in the viscosity sense, and  $w$ be a convex function such that
\begin{equation}\label{touch-below}
w(x_0)=u(x_0), \ \ w(x) \le u(x) \ \ {\rm  in} \ B_\delta.
\end{equation}
For sufficiently small  $\varepsilon\in(0,\varepsilon_0)$ and any function $f$, bounded above, we have
\begin{equation}\label{C1,1-est}
\ell(x) \le w(x) \le \ell(x)+\frac12\,C_\varepsilon\left(\sup_{B_\delta}f^+\right)|x-x_0|^2 \ \ {\rm  in} \ B_{\varepsilon\delta},
\end{equation}
where $\ell(x)$ is the supporting hyperplane for $w$ at $x_0$. In particular, there exists a convex paraboloid of opening $C_\varepsilon/a_n$ touching the graph of $w$ from above.

Here $\varepsilon_0>0$ depends on (a positive lower bound of) $a_n$ and (an upper bound of)  $\hat a_n$ defined in Subsection \ref{Radial_solutions}; moreover $C_\varepsilon\to 1/a_n$ as $\varepsilon \to 0$. Therefore, when $u$ is second order differentiable and $f$ is continuous at $x=x_0$, we get
\begin{equation}\label{lambdan}
\lambda_n^+(D^2w(x_0)) \le\frac1{a_n}\, f^+(x_0).
\end{equation}
\end{lem}
\begin{proof} The first one inequality in (\ref{C1,1-est}) depends on the fact that $\ell(x)$ is the supporting hyperplane of $w$ at $x_0$.

Concerning the second one, we may proceed assuming $x_0=0$ and $\delta=1$. 

\vskip0.2cm

(i) Subtracting $\ell(x)$, we consider the functions  $v(x)=u(x)-\ell(x)$ and $\varphi(x)=w(x)-\ell(x)$, which satisfy in turn the assumptions on $u(x)$ and $w(x)$, respectively. This simplifies the computations, since $\varphi(0)=0$ and the supporting hyperplane for $v(x)$ at $x=0$ is now horizontal, so that $\varphi(x) \ge 0$ in $B_1$.

In this way, we are reduced to show
\begin{equation}\label{red-est}
 \varphi(x) \le \frac12\,C_\varepsilon K |x|^2 \ \ {\rm  in} \ B_{\varepsilon},
\end{equation}
with $K=\sup_{B_1}f^+$, under the assumptions 
\begin{equation}\label{red}
\varphi(0)=v(0)=0, \ \ \varphi(x) \le v(x) \ \ {\rm  in} \ B_1,
\end{equation}
and
\begin{equation}\label{red-D2}
a_1\lambda_1(D^2v)+\dots+a_n\lambda_n(D^2v) \le f^+(x) \ \ {\rm  in} \ B_1,
\end{equation}
which implies
\begin{equation}\label{redd-D2}
\hat a_n\lambda_1(D^2v)+a_n\lambda_n(D^2v) \le f^+(x) \ \ {\rm  in} \ B_1.
\end{equation}

(ii) Let $0 <\rho < \varepsilon$ and $M_\rho$ be the maximum of $\varphi$ on $\overline B_\rho$. We may suppose that a maximum point  is $x_\rho=(0,\dots,0,\rho) \in \partial B_\rho$.  Since the supporting hyperplane for $\varphi(x)$ at $x_\rho$ is constant on the tangent line to $B_\rho$ through $x_\rho$, we have 
\begin{equation}\label{side-rho}
\varphi(x) \ge M_\rho  \ \ {\rm for} \ x=(x', \rho),
\end{equation}
 where $x'=(x_1,\dots,x_{n-1})$.\\
Let us consider now the cylindrical box 
$$R = \{x=(x',x_n) : |x'| < \sqrt{1-\rho^2},\ -\varepsilon \rho <x_n < \rho\} \subset B_1,$$ 
and the paraboloid 
$$P(x)= \frac12(x_n+\varepsilon\rho)^2-\frac12\,\frac{(1+\varepsilon)^2}{1-\rho^2}\,\rho^2|x'|^2.$$
Evaluating $P(x)$ on $\partial R$, when $x_n=-\varepsilon\rho$ or $|x'|=\sqrt{1-\rho^2}$, we have $P(x) \le 0$. On the remaining part of $\partial R$, $x_n=\rho$, we have $P(x) \le \frac12\,{(1+\varepsilon)^2}\rho^2$, from which 
\begin{equation}\label{P-above}
\frac{M_\rho}{\frac12\,{(1+\varepsilon)^2}\,\rho^2}\,P(x) \le \varphi(x) \ \ {\rm  on} \ \partial R.
\end{equation}

(iii) Since $\rho<\varepsilon$, then $P(x)$ is solution of the differential inequality
\begin{equation}\label{P}
\begin{split}
\hat a_n\lambda_1(D^2P)+ a_n\lambda_n(D^2P)& \ge a_n-\frac{(1+\varepsilon)^2}{1-\rho^2}\,\hat a_n\rho^2\\
&\ge a_n-\frac{1+\varepsilon}{1-\varepsilon}\,\hat a_n\varepsilon^2\\
& \equiv a_n- \hat a_nc_\varepsilon,
\end{split}
\end{equation}
where $c_\varepsilon \to 0$ as $\varepsilon \to 0^+$, so that  $a_n-\hat a_nc_\varepsilon>0$ for $\varepsilon<\varepsilon_0$ small enough, and the function $\displaystyle  Q(x)\equiv\frac{K}{a_n-\hat a_nc_\varepsilon}\,P(x)$
satisfies the differential inequality
\begin{equation}\label{Q-subsoln}
\hat a_n\lambda_1\left(D^2Q\right)+ a_n\lambda_n\left(D^2Q\right) \ge K \ge f^+ \ \ \hbox{\rm in } \ B_1.
\end{equation}

(iv) We claim that 
\begin{equation}\label{Mrho}
 M_\rho = \max_{\overline B_\rho}\varphi(x) \le \frac12 \frac{(1+\varepsilon)^2}{a_n- \hat a_nc_\varepsilon}\,K \rho^2
\end{equation}
In fact, arguing by contradiction, suppose that  $M_\rho > \frac12 \frac{(1+\varepsilon)^2}{a_n-\hat a_nc_\varepsilon}\,K \rho^2$. Then using (\ref{P-above}) and (\ref{red}),
$$Q(x)=\frac{K}{a_n-\hat a_nc_\varepsilon}\,P(x) < \frac{M_\rho}{\frac12 \,{(1+\varepsilon)^2}\rho^2}\,P(x) \le \varphi(x) \le v(x) \ \ {\rm  on} \ \partial R.$$
By (\ref{red-D2}) and (\ref{Q-subsoln}), the comparison principle would imply $Q(x) \le v(x)$ in $R$, and this is a contradiction with $v(0)=0 < Q(0)$, which proves the claim.\\

 Setting $\rho=|x|$ in (\ref{Mrho}), as in the proof of \cite[Theorem 3.2]{CC}, we conclude that the statement of the theorem holds with $C_\varepsilon= \frac{(1+\varepsilon)^2}{a_n-\hat a_nc_\varepsilon}$, where $c_\varepsilon \to 0$ as $\varepsilon \to 0^+$.
\end{proof}

\vskip0.2cm

{\bf Proof of Theorem \ref{ABP}: viscosity solutions.}  We follow the lines of the proof of \cite[Theorem 3.6]{CC}, considering subsolutions $u\in$\,usc$(\overline\Omega)$. The case of supersolutions $u\in$\,lsc$(\overline\Omega)$ with the estimate (\ref{ABP-est-super}) from below can be obtained by duality, passing to $-u$.

\vskip0.2cm

Let $\mathcal M_{\bf a} \in \mathcal A$. From Lemma \ref{C1,1} and duality we deduce a similar conclusion for subsolutions $u\in$\,{\rm lsc}$(\overline B_\delta)$, where $B_\delta=\{|x-x_0|<\delta\}$, such that
\begin{equation}\label{D2-below}
\mathcal M_{\bf a}(D^2u) \ge f(x) \ \ {\rm  in} \ B_\delta
\end{equation}
in the viscosity sense. Let $w$ be a concave function such that
\begin{equation}\label{touch-above}
w(x_0)=u(x_0), \ \ w(x) \le u(x) \ \ {\rm  in} \ B_\delta.
\end{equation}
For sufficiently small  $\varepsilon\in(0,\varepsilon_0)$ and any function $f$, bounded above, we have
\begin{equation}\label{C1,1-est}
\ell(x) -\frac12\,C_\varepsilon\left(\sup_{B_\delta}f^+\right)|x-x_0|^2\le w(x) \le \ell(x) \ \ {\rm  in} \ B_{\varepsilon\delta},
\end{equation}
where $\ell(x)$ is the supporting hyperplane for $w$ at $x_0$. In particular, there exists a concave paraboloid of opening $C_\varepsilon$ touching the graph of $w$ from below.

Here $\varepsilon_0>0$ depends on a lower bound for $a_1$ and an upper bound for  $ \hat a_1$ defined in Subsection \ref{Radial_solutions}, and $C_\varepsilon\to 1/a_1$ as $\varepsilon \to 0$. Therefore, when $u$ is second order differentiable and $f$ is continuous at $x=x_0$, we get
\begin{equation}\label{lambda1}
\lambda_1^-(D^2w(x_0)) \le\frac1{a_1}\, f^-(x_0).
\end{equation}

Using  \cite[Lemma 3.5]{CC}, we deduce from the above that $\Gamma_u \in C^{1,1}(\overline B_d)$. Hence $\Gamma_u $ is second order differentiable a.e. in $\overline B_d$ and (\ref{maine}) holds for $\Gamma_u^+$ in $B_d$.

 If $u \le 0$ on $\partial \Omega$ then we have:
\begin{equation}\label{maine-sub}
\sup_{B_d}u^+ \leq C_nd\left(\int_{\Gamma^+_u=u}|\mbox{det} D^2\Gamma_u^+ (x)|\,dx\right)^{1/n}.
\end{equation}
Reasoning as in the proof of \cite[Theorem 3.6]{CC}, that is observing that the upper contact points are in $\Omega$ and  $\Gamma_u^+$ is  second order differentiable a.e. on $\{\Gamma^+_u=u\}$, where $f$ is continuous and therefore, by (\ref{lambda1}):
\begin{equation}\label{det-lambda1}
|\det D^2(\Gamma^+_u(x))| \le  \left(\lambda_1^-(D^2\Gamma^+_u(x))\right)^n \le \frac1{a_1^n}\, (f^-(x))^n.
\end{equation}
 Estimating (\ref{maine-sub}) with (\ref{det-lambda1}), we get the ABP estimate (\ref{ABP-est-sub}) for $u \le0$ on $\partial\Omega$.
 Passing to $u-\sup_{\partial\Omega}u$, which is $\le 0$ on $\partial\Omega$, we conclude that (\ref{ABP-est-sub}) holds.\qed


 \section{Harnack inequality and $C^\alpha$ estimates}\label{harnack}

The Harnack inequality, classically related to the mean properties of the Laplace operator, is a powerful nonlinear technique for regularity in the framework of fully nonlinear equations. We refer to \cite{GT} for solutions of linear uniformly elliptic equations in Sobolev spaces, to \cite{TRU1} for quasi-linear uniformly elliptic equations and to \cite{CAF,CC} for viscosity solutions of fully nonlinear equations.

\vskip0.2cm

In order to prove the Harnack inequality for non-negative solutions and the related local estimates for subsolutions and non-negative supersolutions, respectively known in literature (see for instance \cite{GT}) as the local maximum principle and  the weak Harnack inequality, we could employ the same strategy of \cite[Chapter 4]{CC}. 

A quicker way, sufficient for the applications, is based on the inequalities (\ref{Ma-Pucci}) obtained in Section \ref{aux}. 

The results are given in cubes, and here $\mathcal Q_{\ell}$ is a cube of $\mathbb R^n$ of edge $\ell$ centered at the origin, i.e. $Q_\ell=\{|x_i|<\ell/2, i=1,\dots,n\}$, but they could be equivalently stated in balls.

\begin{thm}\label{lmp} {\rm (local maximum principle)} Let $\mathcal M_{\bf a}\in {\mathcal A}_1$. Let $u$ be a viscosity subsolution of the equation $\mathcal M_{\bf a}(D^2u)=f$ in $Q_1$, where $f$ is continuous and bounded. Then
\begin{equation}\label{lmp-ineq}
\sup_{Q_{1/2}} u \le C_p\left( \|u^+\|_{L^{p}(Q_{2/3})}+ \|f^-\|_{L^n(Q_1)}\right),
\end{equation}
where $C_p$ is a constant depending only on $n$, $p$, $a_1$ and $\tilde a$.
\end{thm}
\begin{proof} In view of inequalities  (\ref{Ma-Pucci}), we have $\mathcal M_{{a_1}/n,\tilde a n}^+(D^2u) \ge \mathcal M_{\bf a}(D^2u) \ge f(x) \ge -f^-(x)$, and therefore we can apply Theorem 4.8 (2) of \cite{CC} to obtain (\ref{lmp-ineq}).
\end{proof}

\begin{thm}\label{w-H} {\rm (weak Harnack inequality)}  Let $\mathcal M_{\bf a}\in \mathcal A_n$. Let $u \ge 0$ be a viscosity supersolution of the equation $\mathcal M_{\bf a}(D^2u)=f$ in $Q_{1}$, where $f$ is continuous and bounded. Then
\begin{equation}\label{w-H-ineq}
\|u\|_{L^{p_0}(Q_{2/3})}\le C_0\left(\inf_{Q_{3/4}} u + \|f^+\|_{L^n(Q_{1})}\right),
\end{equation}
where $p_0>0$ and $C_0$ are universal constants, depending only on $n$, $p$, $a_n$ and  $\tilde a$.
\end{thm}
\begin{proof} In view of inequalities  (\ref{Ma-Pucci}), we have $\mathcal M_{{a_n}/n,\tilde a n}^-(D^2u) \le \mathcal M_{\bf a}(D^2u) \le f(x) \le f^+(x)$, and therefore we can apply Theorem 4.8 (1) of \cite{CC} to obtain (\ref{w-H-ineq}).
\end{proof}

The proof of Theorem \ref{Harnack}  (Harnack inequality) follows at once.

\vskip0.2cm

\noindent {\bf Proof of Theorem \ref{Harnack}}. Let $p_0>0$ be the exponent of Theorem \ref{w-H}. From (\ref{w-H-ineq}) and (\ref{lmp-ineq}) it follows that
\begin{align*} \sup_{Q_{1/2}} u &\le  C_{p_0}\left( \|u\|_{L^{p_0}(Q_{2/3})}+ \|f^-\|_{L^n(Q_1)}\right) \\
&\le C_{p_0}\left( C_0\left(\inf_{Q_{3/4}} u + \|f^+\|_{L^n(Q_1)}\right)+ \|f^-\|_{L^n(Q_1)}\right),
\end{align*}
which yields the result.\qed

\vskip0.2cm

From the Harnack inequality, in a standard way, using the technique for the proof of  \cite[Proposition 4.10]{CC} and  \cite[Lemma 8.23]{GT}, the following H\"older regularity results and $C^\alpha$ interior estimates can be obtained. We give the result with concentric balls $B_1$ and $B_{1/2}$ of radius $1$ and $1/2$, respectively.

\begin{thm} \label{Holder} {\rm (interior H\"older continuity)}  Let $\mathcal M_{\bf a}\in \mathcal A$.  Let $u$ be a viscosity solution of the equation $\mathcal M_{\bf a}(D^2u) =f$ in $B_1$, where $f$ is continuous and bounded. Then $u\in C^\alpha(\overline B_{1/2})$ and 
$$\|u\|_{C^\alpha(\overline B_{1/2})} \le C\left(\|u\|_{L^\infty(B_1)}+\|f\|_{L^n(B_{1})}\right),$$
where $C$ is a positive constant depending only on $n$, $a^*=\min(a_1,a_n)$ and $\tilde a = (a_1+\dots+a_n)/n$.
\end{thm}

Global H\"older estimates can be proved for domains with the uniform exterior sphere condition (S), see Section \ref{intro}, via the boundary H\"older estimates of the lemma below. We adopt the following notations, for the H\"older seminorm ($0<\gamma <1)$) of a function $h: D \to \mathbb R$ in a subset $D$ of $\mathbb R^n$:

\begin{equation}\label{betanorm}
 [h]_{_{\beta,D}} = \sup_{\substack{{x,y \in D}\\ {x\neq y}}}\frac{|h(x)-h(y)|}{|x-y|^\beta};\\ 
\end{equation}

  \begin{lem} \label{Holder-bdary-lem}  Let $\mathcal M_{\bf a}\in \overline{\mathcal A}$ and $u$ be a viscosity solution of the equation $\mathcal M_{\bf a}(D^2u) =f$ in a bounded domain $\Omega$, where $f$ is continuous and bounded. 

We assume that $\Omega$ satisfies a uniform exterior sphere condition (S) with  radius $R>0$, and  $u=g$ on $\partial \Omega$.
\vskip0.2cm
(i) If $g \in C^\beta(\partial \Omega)$ with $\beta \in (0,1]$, then
\begin{equation}\label{Holder-bdary1}
\sup_{\substack{x \in \Omega\\y\in\partial\Omega}}\frac{u(x)-u(y)}{|x-y|^{\beta/2}}\le  C\left([g]_{_{\beta,\partial\Omega}}+\|f^-\|_{L^\infty(\Omega)}\right) 
\end{equation}
with $C>0$  depending only on $n$, $\tilde a$, $R$ and $\beta$.

\vskip0.2cm
(ii) Assume in addition that $\Omega$ has a uniform Lipschitz boundary with Lipschitz constant $L$. If $g \in C^{1,\beta}(\partial \Omega)$ with $\beta \in [0,1)$, then
\begin{equation}\label{Holder-bdary2}
\sup_{\substack{x \in \Omega\\y\in\partial\Omega}}\frac{u(x)-u(y)}{|x-y|^{(1+\beta)/2}}\le  C\left([g]_{_{1,\partial\Omega}}+[Dg]_{_{\beta,\partial\Omega}}+\|f^-\|_{L^\infty(\Omega)}\right) 
\end{equation}
with $C>0$  depending only on $n$, $\tilde a$, $R$, L and $\beta$.

\vskip0.2cm
If $u$ is a viscosity supersolution, (i) and (ii) holds with $u(y)-u(x)$ and $\|f^+\|_{L^\infty(\Omega)}$ instead of  $u(x)-u(y)$ and $\|f^-\|_{L^\infty(\Omega)}$, respectively.

\end{lem}

\begin{proof} 
We treat in detail the case of subsolutions. The result for subsolutions will follow by duality.
\vskip0.2cm

Therefore, suppose that $u \in$\,usc$(\overline \Omega)$ is a subsolution of the equation $\mathcal M_{\bf a}(D^2u) =f$ in $\Omega$  such that $u=g$ on $\partial\Omega$.  

Let $y \in \partial \Omega$ and $B_R$ a ball of radius $R$, centered at $x_0\in\mathbb R^n$, such that $y \in \partial B_R$ and $\overline \Omega\subset \overline B_R$, according to (S).  Supposing, as we may, $y=(0,\dots,0,0)$ and  $x_0=(0,\dots,0,R)$, then $\overline B_R$ is described by the inequality $x_1^2+\dots x_{n-1}^2+ (x_n-R)^2 \le R^2$. 

It follows that 
\begin{equation}\label{fromS}
|x|^2\le 2Rx_n, \ \ x \in\overline\Omega.
\end{equation}

\vskip0.3cm

{\it Case (i)} 

\vskip0.2cm

By assumption on $g$ and (\ref{fromS}), we have
\begin{equation}\label{gi1y}
|g(x)| \le [g]_{\beta,\Omega}\, |x|^{\beta},  \ \ x \in \overline \Omega,
\end{equation}

To simplify, we may suppose:  $g(0)=0$, so that in particular
\begin{equation}\label{gi1}
g(x) \le [g]_{\beta,\Omega} |x|^{\beta} \le (2R)^{\beta/2}[g]_{\beta,\Omega}\,x_n^{\beta/2}, \ \ x \in \overline \Omega.
\end{equation}

Next, we define
\begin{equation}\label{Ctilde}
\varphi(x)=C_1\,([g]_{\beta,\Omega}+\varepsilon)\,x_n^{\beta/2} - \frac1{2|{\bf a}|}\,\|f^-\|_{_{L^\infty(\Omega)}}|x|^2 
\end{equation}
where $\varepsilon$ is any positive number and $C_1\ge(2R)^{\beta/2} + \frac{(2R)^{2-\beta/2}}{2|{\bf a}|}\,\frac{\|f^-\|}{[g]_{_{\beta,\Omega}}+\varepsilon}$. 

Thus from (\ref{gi1}) 
\begin{equation}\label{i-bdary}
u(x) = g(x) \le \varphi(x) \ \ \hbox{\text on}\ \partial\Omega.
\end{equation}
 Moreover, $\varphi$ is a supersolution in $\Omega$:
\begin{equation}\label{i-supersoln}
\begin{split}
\mathcal M_{\bf a}(D^2\varphi) &=a_1\, C_1\, \left([g]_{\beta,\Omega}+\varepsilon\right)\tfrac{\beta}{2}\left(\tfrac{\beta}{2}-1\right)  \,x_n^{\frac{\beta}{2}-2}- \|f^-\|_{_{L^\infty(\Omega)}} \\
&\le  -f^-(x) \ \ \hbox{\text in}\ \Omega.
\end{split}
\end{equation}

By the comparison principle $u(x) \le \varphi(x)$ for all $x \in \Omega$, from which
\begin{equation}\label{Holder-bdary-above-0}
\frac{u(x)}{|x|^{\beta/2}}\le C_2\,\left([g]_{\beta,\Omega}+ \|f^-\|_{_{L^\infty(\Omega)}}\right). 
\end{equation}
Then for an arbitrary $y \in \partial\Omega$ we have
\begin{equation}\label{Holder-bdary-above-y}
\frac{u(x)-u(y)}{|x-y|^{\beta/2}}\le C\,\left([g]_{\beta,\Omega}+ \|f^-\|_{_{L^\infty(\Omega)}}\right),
\end{equation}
from which (\ref{Holder-bdary2}) follows.
\vskip0.3cm

{\it Case (ii)} 

\vskip0.2cm

By assumption on $g$ and on $\partial\Omega$, we have
\begin{equation}\label{g1y}
|g(x) - g(y)- \langle Dg(y), x-y\rangle| \le C_1[Dg]_{\beta,\Omega} |x-y|^{1+\beta},  \ \ x \in \overline \Omega,
\end{equation}
where $C_1$ is a positive constant depending on the Lipschitz constant $L$ for $\partial \Omega$.

We adopt the above simplifications: $y=(0,\dots,0,0)$,  $x_0=(0,\dots,0,R)$, $g(y)=0$, so that in particular
\begin{equation}\label{g1}
g(x) \le \langle Dg(0),x\rangle+C_1[Dg]_{\beta,\Omega} |x|^{1+\beta}, \ \ x \in \overline \Omega.
\end{equation}

 Therefore 
\begin{equation}\label{g2}
g(x) \le \langle Dg(0),x\rangle+C_2[Dg]_{\beta,\Omega}\, x_n^{(1+\beta)/2}, \ \ x \in \overline \Omega,
\end{equation}
where $C_2$ is a positive constant depending on $L$, $R$ and $\beta$.

Next, we define
\begin{equation}\label{Ctilde}
\varphi(x)= \langle Dg(0),x\rangle+C_3\left([Dg]_{\beta,\Omega}+\varepsilon\right) x_n^{(1+\beta)/2} - \frac1{2|{\bf a}|}\,\|f^-\|_{_{L^\infty(\Omega)}},
\end{equation}
where $\varepsilon$ is any positive number and $C_3\ge C_2 + \frac{(2R)^{1-\beta/2}}{2|{\bf a}|}\,\frac{\|f^-\|}{[Dg]_{_{\beta,\Omega}}+\varepsilon}$. 

\vskip0.2cm

Therefore by  (\ref{g2}):
\begin{equation}\label{ii-bdary}
u(x) = g(x) \le \varphi(x) \ \ \hbox{\text on}\ \partial\Omega.
\end{equation}
Moreover $\varphi$ is a supersolution in $\Omega$:
\begin{equation}\label{ii-supersoln}
\begin{split}
\mathcal M_{\bf a}(D^2\varphi) &=a_1 C_3\, \left([Dg]_{\beta,\Omega}+\varepsilon\right)  \tfrac{\beta+1}{2}\tfrac{\beta-1}{2}\,x_n^{\frac{\beta+1}{2}-2}- \|f^-\|_{_{L^\infty(\Omega)}} \\
&\le  -f^-(x) \ \ \hbox{\text in}\ \Omega.
\end{split}
\end{equation}

By the comparison principle, we get $u \le \varphi$ in $\Omega$, and therefore by (\ref{Ctilde}):
\begin{equation}
\begin{split}
\frac{u(x)}{|x|^{(1+\beta)/2}} &\le  |Dg(0)|\,|x|^{(1-\beta)/2}+C_4\left([Dg]_{\beta,\Omega}+ \|f^-\|_{_{L^\infty(\Omega)}}\right) \\
& \le  (4R)^{(1-\beta)/2}\,[Dg]_{_{0,\partial\Omega}}+ C_4\left([Dg]_{\beta,\Omega}+ \|f^-\|_{_{L^\infty(\Omega)}}\right).
\end{split}
\end{equation}

Then for arbitrary $y \in \partial\Omega$ we have
\begin{equation}\label{Holder-bdary-above-y-2}
\frac{u(x)-u(y)}{|x-y|^{(1+\beta)/2}} \le  (4R)^{(1-\beta)/2}\,[Dg]_{_{0,\partial\Omega}}+ C_4\left([Dg]_{\beta,\Omega}+ \|f^-\|_{_{L^\infty(\Omega)}}\right)
\end{equation}
for all $x \in \Omega$, from which (\ref{Holder-bdary2}) follows. 
\end{proof}

We are ready to show the global H\"older estimates of Theorem \ref{Holder-global}.

\vskip0.2cm

\noindent {\bf Proof of Theorem \ref{Holder-global}}. 
 Let $\alpha \in (0,1)$ be the H\"older exponent of Theorem \ref{Holder}. From the boundary H\"older estimates of Lemma \ref{Holder-bdary-lem} we deduce an estimate of type
\begin{equation}\label{bdary-est}
|u(x)-u(y)| \le C(g,f) |x-y|^{\gamma_b}, \  \ x\in\Omega, \ y \in\partial \Omega,
\end{equation}
where $\gamma_{b} = \beta/2$  in the case (i) and $\gamma_{b}  = (1+\beta)/2$  in the case (ii).

\vskip0.2cm 
We want to show a global H\"older estimate with exponent $\gamma=\min(\alpha,\gamma_b)$.
\vskip0.2cm
For proving the result we follow the same lines of \cite[Proposition 4.13]{CC}.

Thus, for $x,y \in \Omega$ we set $d_x=$\,dist$(x,\partial \Omega)=|x-x_0|$, $d_y=$\,dist$(y,\partial \Omega)=|y-y_0|$ for $x_0,y_0\in \partial\Omega$, and suppose $d_y \le d_x$. 

Here the constants $C_i$ will depend at most on $n$, $a^*$, $\tilde a$, $R$, $L$ and $\beta$.

\vskip0.2cm

(i) Suppose $|x-y| \le d_x/2$. Since $y \in \overline B_{d_x/2}(x) \subset B_{d_x}(x) \subset \Omega$, then we can apply Theorem \ref{Holder} properly scaled to the function $u(x)-u(x_0)$, and then the H\"older boundary estimate (\ref{bdary-est}) obtaining
\begin{equation}\label{internal}
\frac{|u(x)-u(y)|}{|x-y|^\alpha}\, d_x^\alpha\le C_1\|u-u(x_0)\|_{L^\infty(B_{d_x}(x))} \le C_2K\,d_x^{\gamma_b}
\end{equation}

Recall that $\gamma \le \alpha$. Since $d_x/|x-y| \ge 2$, from this we get
\begin{equation}\label{internal-2}
\frac{|u(x)-u(y)|}{|x-y|^\gamma}\,d_x^\gamma  \le\frac{u(x)-u(y)}{|x-y|^\alpha}\, d_x^\alpha \le C_2K\, d_x^{\gamma_b}.
\end{equation}
Since also $\gamma \le \gamma_b$,
\begin{equation}\label{internal-end}
\frac{|u(x)-u(y)|}{|x-y|^\gamma} \le C_2K\,d_x^{\gamma_b -\gamma} \le C_2K\,d^{\gamma_b -\gamma}\equiv  C_3K\,.
\end{equation}

\vskip0.2cm
(ii) Suppose now $|x-y| \ge d_x/2$. Since $d_y \le d_x \le 2|x-y|$ and $|x_0-y_0| \le d_x+|x-y|+d_y$, then by (\ref{bdary-est})
\begin{equation}\label{bdary-points}
\begin{split}
|u(x)-u(y)| &\le |u(x)-u(x_0)|+|u(x_0)-u(y_0)|+|u(y_0)-u(y)|\\
& \le C_4K(d_x^{\gamma_b}+|x_0-y_0|^{\gamma_b}+d_y^{\gamma_b})  \le C_5K |x-y|^\gamma.
\end{split}
\end{equation}

 From (\ref{internal-end}) and (\ref{bdary-points}), letting $\overline C=\max(C_3,C_5)$, we deduce the desired estimate:
\begin{equation}\label{seminorm}
[u]_{\gamma,\Omega} \le \overline CK.
\end{equation}
\qed

\vskip0.2cm

In some cases, when the weights $a_i$ are concentrated near the one of the extremal eigenvalues, we obtain an explicit interior H\"older exponent. 

\begin{lem}\label{Holder-asy-int} Let $\mathcal M_{\bf a} \in \overline{\mathcal A}$ be such that   either   $a_1 \ge \hat a _1$ (resp. $a_n \ge \hat a _n$).

Suppose that  $u\in usc(B_1)$ (resp. $u\in lsc(B_1)$)  is a viscosity subsolution (resp. supersolution) of the equation $\mathcal M_{\bf a}(D^2u) =f$  in $B_1$, a ball of radius $1$, and $f$ is continuous and bounded above (resp. below) in $B_1$.

Then $u \in C^\alpha(B_1)$ and the following interior $C^\alpha$ estimate holds:
\begin{equation}\label{asy1}
[u]_{\alpha,B_{1/2}}\le C\left( \|u\|_{L^\infty(B_1)}+ \|f^-\|_{L^\infty(B_1)}\right),
\end{equation}
resp.
\begin{equation}\label{asy2}
[u]_{\alpha,B_{1/2}}\le C\left( \|u\|_{L^\infty(B_1)}+ \|f^+\|_{L^\infty(B_1)}\right),
\end{equation}
 where $B_{1/2}$ is a ball of radius $1/2$ concentric with  $B_1$, 
$$\alpha = 1-\hat a_1/a_1 \ \  (\hbox{\text resp.} \ \alpha = 1-\hat a_n/a_n),$$
 and $C$ a positive costant depending on $n$, $a_1$ and $\hat a_1$ (resp. $a_n$ and $\hat a_n$). 
\end{lem}
\begin{proof}
We only treat the case of subsolution, when $a_1\ge \hat a_1$. The case of supersolutions, when  $a_n\ge \hat a_n$,  will follow by duality. 

We assume that  the balls $B_1$ and $B_{1/2}$ are centered at $0$. Then we take $x',x'' \in B_{1/2}$, and consider the ball $B_{1/2}(x')$. We note that  on $\partial B_{1/2}(x')$:
\begin{equation}\label{asy1/2}
u(x)-u(x') \le 2\|u\|_{L^\infty(B_1)} \le 2^{1+\alpha}\|u\|_{L^\infty(B_1)}\,|x-x'|^\alpha
\end{equation}
Next, we define 
\begin{equation}\label{phi-Holder-int}
\varphi(x)= C_1\,\|u\|_{L^\infty(B_1)}\,|x|^\alpha-\frac{1}{2|{\bf a}|}\,\|f^-\|_{L^\infty(B_1)}\,|x|^2
\end{equation}
where $C_1 =2^{1+\alpha}+ \frac{1}{2|{\bf a}|}\,\frac{\|f^-\|}{\|u\|}$ (in the nontrivial case $u\not\equiv0$).

Thus on $\partial B_{1/2}(x')$:
\begin{equation}\label{asy1/2}
\begin{split}
u(x)-u(x') &\le 2^{1-\alpha}\|u\|_{L^\infty(B_1)}\,|x-x'|^\alpha\\
& \le \varphi(x-x').
\end{split}
\end{equation}

On the other hand,  
\begin{equation}\label{asy-phi}
\begin{split}
\mathcal M_{\bf a}(D^2\varphi(x-x')) &= C_1\left(a_1(\alpha-1)+\hat a_1\right)|x-x'|^{\alpha-2}\\
&-\|f^-\|_{L^\infty(\Omega)}\le -f^-(x).
\end{split}
\end{equation}

By (\ref{asy-phi}) and (\ref{asy1/2}), using the comparison principle  we get $u(x)-u(x') \le \varphi(x-x')$ in $B_{1/2}(x')$, from which in particular:
\begin{equation}\label{asy1}
u(x'')-u(x') \le C\left(\|u\|_{L^\infty(B_1)}+\|f^-\|_{L^\infty(B_1)}\right)|x''-x'|^\alpha.
\end{equation}
Interchanging the role of $x'$ and $x''$, we get (\ref{asy1}).
\end{proof} 

\vskip0.2cm

Combining Lemma \ref{Holder-bdary-lem} with Lemma \ref{Holder-asy-int}, we obtain the global estimates of Theorem \ref{Holder-asy-global}.

\vskip0.2cm

{\bf Proof of Theorem \ref{Holder-asy-global}}. To obtain (\ref{global-Holder1-intro}) and (\ref{global-Holder2-intro}) it  is sufficient to follow the proof of Theorem \ref{Holder-global}. 

The above estimates (\ref{global-Holder1-asy}) and (\ref{global-Holder2-asy})  are in particular obtained from this proof  taking $\gamma=\alpha$.  

We use once more the boundary H\"older estimates of Lemma \ref{Holder-bdary-lem}, with $\beta=2\alpha$ for (\ref{global-Holder1-asy}) and $\beta=2\alpha-1$ for (\ref{global-Holder2-asy}), as there. But we use here the interior $C^\alpha$ estimates of Lemma \ref{Holder-asy-int}, instead of Theorem \ref{Holder}.
\qed

\vskip0.2cm

\begin{rem}\label{Lip}
Note that  Theorem \ref{Holder-asy-global} provides Lipschitz estimates only in the case $\hat a_1=0$ and $\hat a_n=0$, corresponding to the operators $\mathcal M_{{\bf e}_1}[u]=\lambda_1[u]$ and $\mathcal M_{{\bf e}_n}[u]=\lambda_n[u]$. See for instance \cite{BGI}.
\end{rem}

\vskip0.2cm

\begin{rem}\label{worseandworse}
Asking for higher regularity of viscosity solutions, we cannot expect viscosity solutions more regular than $C^2.$ Indeed,  we may consider  the function $u:=x_1^2+\omega(x_2)-x_3^2,$ in $\mathbb{R}^3$,  where $\omega$ is a $C^2$ function but no more regular. The same regularity holds for $u$. Assuming in addition $|\omega''(x_2)|<2$, by a straightforward computation we get
$$
D^2u(x)=\mbox{Diag}[(2,\omega''(x_2),-2)],
$$
so that $\mathcal{M}(D^2u(x)=0.$  So we have found a solution $u\in C^2,$ which does not belong to any $C^{2,\beta}$ space, $\beta\in (0,1)$. \end{rem}

\section{The strong maximum principle}\label{strong}

The strong maximum principle for an elliptic operator $F$, such that $F(0)=0$, means that a subsolution of the equation $F[u]=0$ in an open set $\Omega$ cannot have a maximum at a point of $\Omega$ unless to be constant. 

Analogously, the strong minimum principle means that a supersolution $u$ cannot have a minimum at a point of $\Omega$ unless $u$ is constant.

\vskip0.2cm

One of the most elegant proof of the strong maximum principle, also known for this reason as the celebrated Hopf maximum principle \cite{HOP}, is based on boundary point lemma, which we establishes here below for the class of weigthed partial trace operators $\mathcal M_{\bf a}$. To obtain a strong maximum principle it is sufficient to state this lemma just for a ball.

\vskip0.2cm
\begin{lem}\label{Hopf} {\rm (Hopf boundary point lemma)} Let $u\in$\,{\rm usc}$(\overline B)$ be a viscosity subsolution of the equation $\mathcal M_{\bf a}[u]=0$ in a ball $B$, with $M_{\bf a}\in \mathcal A_1$. Let $x_0 \in \partial B$. If $u(x_0)>u(x)$ for all $x \in B$, then the outer  normal derivative of $u$ at $x_0$, if it exists, satisfies the strict inequality 
\begin{equation}\label{Hopf-eq-max}
\frac{\partial u}{\partial \nu}(x_0)>0.
\end{equation}
On the other hand, let $u\in$\,{\rm lsc}$(\overline B)$ be a viscosity supersolution of the equation $\mathcal M_{\bf a}[u]=0$ in a ball $B$, with $M_{\bf a}\in \mathcal A_n$.  If $u(x_0)<u(x)$ for all $x \in B$, then the outer  normal derivative of $u$ at $x_0$, if it exists, satisfies the strict inequality 
\begin{equation}\label{Hopf-eq-min}
\frac{\partial u}{\partial \nu}(x_0)<0.
\end{equation}
\end{lem}

\begin{proof} We just prove the theorem for subsolutions.

We may suppose that $B$ is centered at the origin, i.e  $B = \{|x|<R\}$ for $R>0$. Arguing as in \cite[Section 3.2]{GT}, and considering $0 < \rho <R$, we introduce the radial test function $v(x)=e^{-\alpha r^2} - e^{-\alpha R^2}$, with $r=|x|$. 

By direct computation, see Remark \ref{radial} in Section \ref{aux}, we get
\begin{equation}
\mathcal M_{{\bf a}'}(D^2v) \ge 2\alpha \left(a_1(2\alpha \rho^2-1)-\sum_{i=2}^na_i \right)e^{-\alpha r^2}
\end{equation}
 for $ r\ge \rho$ and $\mathcal M[v]\ge 0$ for $\alpha>0$ large enough.

Since $u(x_0)-u(x)>0$ on $|x|=\rho$, there is a constant $\varepsilon>0$ such that  $u(x_0) - u(x)-\varepsilon v(x)\ge 0$ on $|x|=\rho$, as well as on $|x|=R$. Therefore $\varepsilon v(x) \le u(x_0) - u(x) $ on the boundary of the annulus  $A_{\rho,R}= \{\rho < |x| < R\}$. 

By the comparison principle, the same inequality holds in $A_{\rho,R}$ . In fact $\mathcal M_{{\bf a}'}[\varepsilon v]=\varepsilon \mathcal M_{{\bf a}'}[v] \ge 0$, by positive homogeneity, and by duality $\mathcal M_{{\bf a}'}[u(x_0)-u] \le 0$, so that $\varepsilon v$ and $u(x_0)-u$ are respectively a subsolution and a supersolution in $A$, and we can apply Theorem \ref{comparison} to deduce that
$$
u(x_0)-u(x) \ge \varepsilon v(x) \ \ \hbox{\text for all} \ x \in A.
$$

Taking $x=x_0-t \frac {x_0}{R}$ in the latter inequality, dividing by $t>0$ and letting $t \to 0^+$, we get 
$$
\frac{\partial u}{\partial \nu} (x_0) \ge -\varepsilon \frac{d}{dr}\left( e^{-\alpha r^2}\right)\Big|_{r=R} = 2  \varepsilon \alpha R e^{-\alpha R^2},
$$
which proves (\ref{Hopf-eq-max}).
\end{proof}

\vskip0.2cm
Following \cite{GT} we remark that, whether or not the normal derivative exists, we have instead of (\ref{Hopf-eq-max}) and (\ref{Hopf-eq-min}), respectively, the inequalities
\begin{equation}\label{Hopf-eq2}
\liminf_{\substack{{x\to x_0}\\{x \in \Sigma}}} \frac{u(x_0)-u(x)}{|x-x_0|}>0,
\end{equation}
and 
\begin{equation}\label{Hopf-eq2}
\liminf_{\substack{{x\to x_0}\\{x \in \Sigma}}} \frac{u(x_0)-u(x)}{|x-x_0|}<0,
\end{equation}
where $\Sigma$ is any circular cone of vertex $x_0$ and opening less than $\pi$ with axis along the normal direction at the boundary point $x_0$.

\vskip0.2cm

 The Hopf boundary point lemma can be used to prove the strong maximum principle for classical subsolutions or viscosity subsolutions which are differentiable. A strong maximum principle, valid also for nonsmooth viscosity solutions, can be  obtained through the weak Harnack inequality of Lemma \ref{w-H}.

For a detailed discussion on the strong maximum principle, we refer to the paper \cite{PS} and the papers quoted therein. In the case of fully nonlinear elliptic operators, see for instance \cite{BDL,BDL2}.

\vskip0.2cm

\begin{thm}\label{smp} {\rm(strong maximum principle)} Let $u$ be a non-negative continuous viscosity supersolution of the equation $\mathcal M_{\bf a}(D^2u)=0$, with $\mathcal M_{\bf a} \in \mathcal A_n$, in a domain $\Omega$ of $\mathbb R^n$.  If $u$ has a minimum $m$  at some point  $x_0 \in \Omega$, then $u \equiv m$ in $\Omega$. \\
Similarly, let $u$ be a continuous viscosity subsolution of the equation $\mathcal M_{\bf a}(D^2u)=0$, with $\mathcal M_{\bf a} \in \mathcal A_1$. If $u$  has a maximum $M$ at $x_0\in \Omega$, then $u\equiv M$ in $\Omega$.
\end{thm}
\begin{proof} For the proof in the case of differentiable solutions $u$, based on the Hopf lemma, we refer to the proof of \cite[Theorem 3.5]{GT}.

Concerning viscosity supersolutions (strong minimum principle), let $A=\{x \in \Omega : u(x)=m\}$ and $B=\Omega \backslash A$, so that $A \cup B =\Omega$, $A\cap B= \emptyset$ with $A \neq \emptyset$ and $B$ is open. Moreover, we claim that $A$ is also open. Recalling that $\Omega$ is a open connected set, then $B=\emptyset$, otherwise we would have a contradiction. Then $\Omega=A$, and the first part of the theorem is proved.

We are left with proving that $A$ is open. Let $x_0 \in A$, that is $u(x_0)=m$, and suppose that the cube $Q_\ell$ of side $\ell$ centered at $x_0$ is contained in $\Omega$. By the weak Harnack inequality (\ref{w-H-ineq}), properly scaled and applied to $u-m\ge 0$, we have
\begin{equation}\label{w-H-m}
\|u-m\|_{_{L^{p_0}(Q_{2\ell/3})}}\le C_0\,\inf_{Q_{3\ell/4}}(u-m) =u(x_0)-m=0\,.
\end{equation}
The $u(x)-m$ is constant in $Q_{2/3}$ and by continuity $u(x)-m=u(x_0)-m=0$ for all $x \in Q_{2/3}$, so that $Q_{2/3}\subset A$. This shows that $A$ is open, thereby proving the claim and concluding the proof of the first part.

\vskip0.2cm
In the case of viscosity subsolutions (strong maximum principle), we argue in a similar manner, considering the set $A=\{x \in \Omega : u(x)=M\}$, using the duality and applying the weak Harnack inequality to the supersolution $M-u(x)$ of the dual equation.
\end{proof}

\vskip0.2cm
It follows that for elliptic operators $\mathcal M_{\bf a} \in \mathcal A$ both the strong maximum and minimum principle are satisfied.

\vskip0.2cm

It is plain that the strong maximum principle implies the weak maximum principle (see Section \ref{prelim}) in bounded domains. This is no more true in unbounded domains, where the strong maximum principle may hold while the weak maximum principle fails to hold. An elementary example of this fact is given by the function $u(x_1,x_2)=x_1x_2$, which is harmonic in the whole plane, and therefore satisfies the strong maximum principle in all domains of $\mathbb R^2$, but is positive in the quarter plane $\Omega=\mathbb R_+\times \mathbb R_+$ and zero on $\partial \Omega$, so that the weak maximum principle does not hold in $\Omega$.

\vskip0.2cm

Turning to bounded domains, as observed in Section \ref{prelim},  it is sufficient that $\mathcal M_{\bf a} \in \overline{\mathcal A}$ to have both the weak maximum and minimum principle. Theorem \ref{smp} requires instead $\mathcal M_{\bf a} \in \mathcal A$ to have the strong maximum and minimum principle hold together.

\vskip0.2cm

Actually, the strong maximum and minimum principle may fail when  $\mathcal M_{\bf a} \in \overline{\mathcal A}$, but $\mathcal M_{\bf a} \not\in\mathcal A$. In fact, let us consider the partial trace operator $\mathcal P^+_{k}$ defined above for $1 \le k \le n-1$:
the non-constant function $u(x)=1+\sin x_1$  has a maximum $M=2$  inside the cube $]0,\pi[^n$, even though  $\mathcal P^+_{k}(D^2u)=0$ in $]0,\pi[^n$. 

Similarly, $u(x)$  is non-negative  in the cube $]-\pi,0[^n$ and has a zero inside,  even though $\mathcal P^-_{k}(D^2u)=0$ in the cube $]-\pi,0[^n$  for $1 \le k \le n-1$.

From the proof of Theorem \ref{smp}, the weak Harnack inequality, which would imply the strong minimum principle, fails to hold in general for the partial trace operator $\mathcal P_k^-$ as soon as $k<n$.

Analogously, the Harnack inequality, which would imply both the strong maximum and minimum principle, fails to hold in general for the partial trace operators $\mathcal P_k^\pm$ as soon as $k<n$

 \section{Liouville theorems}\label{liouville}

A direct application of the Harnack inequality yields in a standard fashion the following Liouville result for entire solutions, defined in the whole $\mathbb R^n$.  See for instance \cite{ARV}.
 
\begin{thm} \label{Liouville} {\rm (Liouville theorem)} Let $\mathcal M_{\bf a} \in \mathcal A$. If $u$ is an entire viscosity solution of the equation $\mathcal M(D^2u)=0$  which is bounded above or below, then $u$ is constant.
\end{thm}

It is well known that the above Liouville theorem holds in a stronger unilateral version for the Laplace operator in dimension $n=2$ where instead of solutions, bounded above or below, we may consider subsolutions bounded above and supersolutions bounded below. This is due to the fact that the fundamental solutions are of logarithmic type.  See \cite[Theorem 29]{PW}.

On the other hand, this is no longer true in higher dimension. For instance, the function
\begin{equation}\label{fun}
u(x)=\left\{\begin{array}{cc}
-\frac18\left(15-10|x|^2+3|x|^4\right) & \hbox{\rm for} \ |x|\le1\\
-1/|x| & \hbox{\rm for} \ |x|>1
\end{array}\right.
\end{equation}
is a non-constant subharmonic, bounded function in $\mathbb R^3$.  See \cite[Ch.2, Section 12]{PW}.

As well, the unilateral Liouville theorem does not hold for general elliptic operators even in dimension $n=2$. Actually, as soon as $\lambda < \Lambda$ we can find subsolutions $u$, bounded above, of the equation $\mathcal M^{+}_{\lambda,\Lambda} (D^2u)=0$ in $\mathbb R^2$. For instance, the function (\ref{fun}), regarded as a function of $x \in \mathbb R^2$, is a subsolution of the equation $\mathcal M^{+}_{\lambda,2\lambda} (D^2u)=0$ in $\mathbb R^2$. 

Therefore, the uniform ellipticity is not sufficient by itself to guarantee such an unilateral Liouville property, even in dimension $2$.

However, for particular uniformly elliptic operators  as the minimal Pucci operators $\mathcal M^{-}_{\lambda,\Lambda}$, which are suitably smaller than the Laplace operator, precisely when $n \le 1+\Lambda/\lambda$, the Liouville property still holds for subsolutions, bounded above (see \cite{CL}).  We thank Dr. Goffi for drawing our attention to the latter issue during a workshop where the results of this paper have been announced for the first time\footnote{3 Days in Evolution PDEs, 2019 June 21st}.
\vskip0.2cm

We notice here that the same is true for the min-max operator $\mathcal M(X)=\lambda_1(X)+\lambda_n(X)$, and more generally for the operators $\mathcal M_{\bf a}\in \overline{\mathcal A}$ such that $a_1=\hat a_1$. In fact, from Remark \ref{radial}, considering $0<r_1<r_2$ and setting $M(r)=\max_{B_{r}}u$, the function
$$
\varphi(x) = \frac{M(r_1)\log(r_2/|x|)+M(r_2)\log(|x|/r_1)}{\log(r_2/r_1)},
$$
 satisfies the equation $\mathcal M(D^2\varphi)=0$ as linear combination of a constant and $\log |x|$ with non-negative coefficients, by positive homogeneity. 

 Moreover $u(x) \le \varphi(x)$ on the boundary of the annulus $A_{r_1,r_2}=\{r_1 < |x| < r_2\}$.
\vskip0.2cm
From the comparison principle (Theorem \ref{comparison}) in $A_{r_1,r_2}$ then we obtain the {\bf Hadamard Three-Circles Theorem}, according to which $M(r)$ is a convex function of $\log r$:
\begin{equation}\label{3C}
M(r) \le  \frac{M(r_1)\log(r_2/r)+M(r_2)\log(r/r_1)}{\log(r_2/r_1)}, \ \ r_1 \le r \le r_2
\end{equation}

From (\ref{radialsoln}) the same inequality (\ref{3C}) continues to hold for $\mathcal M_{\bf a}$ such that $a_1=\hat a_1$, so that the same Liouville theorem as for the Laplace operator in dimension $n=2$ holds for viscosity subsolutions $u$, bounded above, of the equation $\mathcal M(D^2u)=0$ in all $\mathbb R^n$ . 

\begin{thm} \label{Liouville} {\rm (unilateral Liouville property)} Let $\mathcal M_{\bf a}\in \overline{\mathcal A}$ be such that $a_1=\hat a_1$. Let $u$ be a viscosity subsolution of the equation $\mathcal M(D^2u)=0$ in $\mathbb R^n\backslash\{0\}$, which is bounded above. Then $u$ is constant.
If $u$ is a subsolution in the whole $\mathbb R^n$ bounded above, then the same conclusion holds if $a_1>\hat a_1$. 

On the other hand, suppose  $\mathcal M_{\bf a}\in \overline{\mathcal A}$ such that $a_n=\hat a_n$. Let $v$ be a viscosity supersolution of the equation $\mathcal M(D^2v)=0$ in $\mathbb R^n\backslash\{0\}$ which is bounded below. Then $v$ is constant.
If $u$ is a supersolution in the whole $\mathbb R^n$, bounded below, then the same conclusion holds if $a_n>\hat a_n$. 
\end{thm}
{\bf Proof.} Let $\mathcal M_{\bf a}\in \overline{\mathcal A}$ with $a_1= \hat a_1$. Reasoning as in \cite[Section 12]{PW}, we take alternatively the limits as $r_1 \to 0^+$ and $r_2 \to \infty$ in (\ref{3C}). So we get
$$M(r) \le M(r_2) \ \ \hbox{\text for} \ r \le r_2, \ \ M(r) \le M(r_1) \ \ \hbox{\text for} \ r \ge r_1$$
concluding that $M(r_1)=M(r_2)$ for arbitrary pairs of positive numbers $r_1, r_2$.

Then $M(r)$ is constant, and by the strong maximum principle $u$ is in turn a constant function. 

Supposing $a_1>\hat a_1$,  for any arbitrary $x_0 \in \mathbb R^n$ we set $v(x)=u(x_1)+C|x-x_1|^{\gamma_2}$, with $\gamma_2=1-\hat a_1/a_1$ and $C\ge0$ to be determined, recalling that by (\ref{radialsoln}) we  have $\mathcal M_{\bf a}(D^2v)=0$ in $\mathbb R^n\backslash\{0\}$. 

We will compare the entire subsolution bounded above, say $u(x) \le M$, in every punctured ball $B_R(x_1)\backslash\{x_1\}$, noting that $u(x_1)=v(x_1)$ and on  $\partial B_R(x_1)$ we have 
\begin{equation}\label{liou-comp}
u(x) \le M \le u(x_1)+C |x-x_1|^{\gamma_2}
\end{equation}
choosing $C=(M-u(x_1)R^{-\gamma_2}$. 

Using the comparison principle (Theorem \ref{comparison}) we infer that this inequality holds in $B_R(x_1)$. Letting $R \to \infty$, we will have $u(x) \le u(x_1)$ for all $x \in \mathbb R^n$. The same holds true for any other $x_2 \in\mathbb R^n$ so that $u(x_1)=u(x_2)$ for all $x_1,x_2 \in \mathbb R^n$.
\vskip0.2cm

Concerning supersolutions $v$, bounded below, of the same equation in $\mathbb R^n$, it is sufficient to note that by duality  the function $u=-v$  is a subsolution, bounded above, of the equation $\mathcal M_{{\bf a}'}[u]=0$, where ${\bf a}'=(a_n,a_{n-1},\dots,a_1)$, in $\mathbb R^n$, and then to use the result proved before for subsolutions. \qed 

\end{document}